\documentclass[11pt,a4paper]{article}

\usepackage{cmap}
\usepackage[T1]{fontenc}

\usepackage{tocloft}

\usepackage{amsmath,amsthm,amssymb,latexsym,graphicx,mathrsfs}
\usepackage{secdot}
\sectiondot{subsection}
\usepackage{a4wide}
\usepackage{float}

\usepackage{mathtools}
\mathtoolsset{showonlyrefs}
\allowdisplaybreaks

\usepackage{marginnote}

\usepackage{xcolor}
\usepackage{url}
\usepackage{hyperref}
\definecolor{ForestGreen}{rgb}{0.1,0.6,0.05}
\definecolor{EgyptBlue}{rgb}{0.063,0.1,0.6}
\definecolor{RipeOlive}{HTML}{556B2F}
\hypersetup{
	colorlinks=true,
	linkcolor=EgyptBlue,         
	citecolor=ForestGreen,
	urlcolor=RipeOlive
}

\usepackage[hyperpageref]{backref}
\usepackage{epstopdf}
\newcounter{dummy}
\usepackage{enumitem}
\makeatletter
\newcommand\myitem[1][]{\item[#1]\refstepcounter{dummy}\def\@currentlabel{#1}}
\makeatother

\newtheorem{theorem}{Theorem}
\newtheorem{proposition}[theorem]{Proposition}
\newtheorem{lemma}[theorem]{Lemma}
\newtheorem{corollary}[theorem]{Corollary}

\theoremstyle{definition}

\newtheorem{remark}[theorem]{Remark}

\numberwithin{equation}{section}
\numberwithin{theorem}{section}

\numberwithin{equation}{section}
\numberwithin{theorem}{section}

\let\OLDthebibliography\thebibliography
\renewcommand\thebibliography[1]{
	\OLDthebibliography{#1}
	\setlength{\parskip}{1pt}
	\setlength{\itemsep}{1pt plus 0.3ex}
}

\newenvironment{proof*}[1]{\begin{trivlist}\item[\hskip%
		\labelsep{{\bf Proof of \/{\rm\bf #1.}}~}]\rm}%
	{\hfill\qed\rm\end{trivlist}}

\newcommand{\W}{W_0^{1,p}}
\newcommand{\Wq}{W_0^{1,q}}
\newcommand{\Wr}{W_0^{1,r}}

\newcommand{\intO}{\int_\Omega}
\newcommand{\intOe}{\int_{\Omega_\varepsilon}}
\newcommand{\intOo}{\int_{\Omega_0}}

\newcommand{\E}{E_{\alpha,\beta}} 
\newcommand{\N}{\mathcal{N}_{\alpha,\beta}} 
\newcommand{\EN}{\widetilde{E}_{\alpha,\beta}} 
\newcommand{\J}{J_{\alpha,\beta}}
\newcommand{\F}{F_{\alpha,\beta}} 
\newcommand{\Ha}{H_{\alpha}} 
\newcommand{\Gb}{G_{\beta}}

\title{
	\vspace*{-1cm}
	Abstract multiplicity results for 
	$(p,q)$-Laplace equations\\ with two parameters 
} 
\author{Vladimir Bobkov ~\&~ Mieko Tanaka \\}
\date{}

\AtEndDocument{%
	\par
	\bigskip
	\bigskip
	\begin{tabular}{@{}l@{}}%
		(V.~Bobkov)\\[0.2em]
		\textsc{Institute of Mathematics, Ufa Federal Research Centre, RAS}\\
		\textsc{Chernyshevsky str. 112, Ufa 450008, Russia}\\[0.3em]
		\textit{E-mail address}: \texttt{bobkov@matem.anrb.ru}\\[1.5em]
		
		(M.~Tanaka)\\[0.2em]
		\textsc{Department of  Mathematics, Tokyo University of Science}\\
		\textsc{Kagurazaka 1-3, Shinjyuku-ku, Tokyo 162-8601, Japan}\\[0.3em]
		\textit{E-mail address}: \texttt{miekotanaka@rs.tus.ac.jp}
\end{tabular}}

\begin{document}
	\maketitle

	\begin{abstract}
		We investigate the existence and multiplicity of abstract weak solutions of the equation $-\Delta_p u -\Delta_q u=\alpha |u|^{p-2}u+\beta |u|^{q-2}u$ in a bounded domain under zero Dirichlet boundary conditions, assuming $1<q<p$ and $\alpha,\beta \in \mathbb{R}$. 
		We determine three generally different ranges of parameters $\alpha$ and $\beta$ for which the problem possesses a given number of distinct pairs of solutions with a prescribed sign of energy.
		As auxiliary results, which are also of independent interest, we provide alternative characterizations of variational eigenvalues of the $q$-Laplacian using narrower and larger constraint sets than in the standard minimax definition.	
		
		\par
		\smallskip
		\noindent {\bf  Keywords}: 
		$p$-Laplacian; $(p,q)$-Laplacian; variational eigenvalues;  multiplicity; symmetric mountain-pass theorem; Nehari manifold.
		
		\smallskip
		\noindent {\bf MSC2010}: 
		35J92,	%Quasilinear elliptic equations with $p$-Laplacian
		35B30,	%Dependence of solutions to PDEs on initial and/or boundary data and/or on parameters of PDEs
		35A01,   %Existence problems for PDEs: global existence, local existence, non-existence
		35B38.	%Critical points of functionals in context of PDEs 
	\end{abstract}
	
	%ToC
	\begin{quote}	
		\tableofcontents	
		\addtocontents{toc}{\vspace*{-2ex}}
	\end{quote}

	\section{Introduction}\label{sec:intro}
	We investigate the quasilinear eigenvalue type problem
	\begin{equation*}\label{eq:D}
		\tag{$\mathcal{D}$} 
		\left\{
		\begin{aligned}
			-&\Delta_p u -\Delta_q u=\alpha |u|^{p-2}u+\beta |u|^{q-2}u 
			&&{\rm in}\ \Omega, \\
			&u=0 &&{\rm on }\ \partial \Omega,
		\end{aligned}
		\right.
	\end{equation*}
	where $1<q<p<+\infty$, $\Delta_r$ with $r \in \{p,q\}$ stands for the $r$-Laplace operator which can be defined as $\Delta_r u ={\rm div}\,(|\nabla u|^{r-2}\nabla u)$ for sufficiently regular functions, and the parameters $\alpha, \beta$ are real numbers.
	In view of the symbolic symmetry of the equation in \eqref{eq:D} and since we are interested in the nonhomogeneous case $q \neq p$, the assumption $q<p$ is imposed without loss of generality.
	We assume by default that $\Omega$ is a bounded domain in $\mathbb{R}^N$, $N \geq 1$.
	Moreover, everywhere except for Appendices~\ref{sec:appendixA} and~\ref{sec:appendixB} we assume, in addition, that $\Omega$ is $C^{1,\kappa}$-regular for some $\kappa \in (0,1)$ when $N \geq 2$. 
	
	The pointwise form of the problem \eqref{eq:D} is presented only for visual clarity and, rigorously, \eqref{eq:D} has to be understood in the weak sense. 
	More precisely, by a (weak) solution of the problem \eqref{eq:D} we understand a critical point of the 
	energy functional $\E \in C^1(\W;\mathbb{R})$ defined as
	\begin{equation*}\label{def:E} 
		\E(u) = \frac{1}{p}\, H_\alpha (u)+\frac{1}{q}\,G_\beta(u), 
	\end{equation*}
	where $W_0^{1,p} = W_0^{1,p}(\Omega)$ is the Sobolev space, and we denote
	\begin{equation*}\label{def:HG}
		H_\alpha (u) =\|\nabla u\|_p^p -\alpha\|u\|_p^p 
		\quad {\rm and}\quad 
		G_\beta(u) =\|\nabla u\|_q^q -\beta\|u\|_q^q.
	\end{equation*}
	Hereinafter, taking $r \in [1,+\infty]$, $\| \cdot \|_r$ stands for the usual norm of the Lebesgue space $L^r=L^r(\Omega)$ and $\| \nabla (\cdot) \|_r$ is that of $W_0^{1,r}$. 
	(We use the latter notation instead of a more accurate variant $\| |\nabla (\cdot)| \|_r$, for brevity.)
	Note that $\W$ is the natural energy space for \eqref{eq:D} in view of the assumption $q<p$ and the boundedness of $\Omega$.
	
	The problem \eqref{eq:D} can be seen as a perturbation of the homogeneous eigenvalue problem for the $p$-Laplacian
	\begin{equation}\label{eq:EP}
		\left\{
		\begin{aligned}
			-&\Delta_p u =\alpha |u|^{p-2}u 
			&&{\rm in}\ \Omega, \\
			&u=0 &&{\rm on }\ \partial \Omega,
		\end{aligned}
		\right.
	\end{equation}
	by a specific second-order term $\Delta_q u+\beta |u|^{q-2}u$ corresponding to the homogeneous eigenvalue problem for the $q$-Laplacian. 
	From this point of view, it is natural to anticipate that the relation between the parameters $\alpha$, $\beta$ and eigenvalues of the $p$- and $q$-Laplacians plays a significant role in the existence theory for \eqref{eq:D}. 
	In particular, the perturbation $\Delta_q u+\beta |u|^{q-2}u$, viewed in the weak (integral) form, is \textit{sign-definite} when $\beta < \lambda_1(q)$ and \textit{sign-indefinite} when $\beta > \lambda_1(q)$, which leads to a significant difference in the structure of the solution set of \eqref{eq:D} in these cases.
	Here, $\lambda_1(q)$ is the first eigenvalue of $-\Delta_q$, see \eqref{eq:lambda1r} in Section~\ref{sec:notations} below.
	Let us also observe that the function $t \mapsto \E(tu)$ has at most one critical point in $(0,+\infty)$ for any $u \in \W \setminus \{0\}$.
	In these respects, the problem \eqref{eq:D} is a sibling to a somewhat better known problem
	\begin{equation*}\label{eq:I}
		\left\{
		\begin{aligned}
			-&\Delta_p u = \alpha |u|^{p-2}u + f(x) |u|^{q-2}u  
			&&{\rm in}\ \Omega, \\
			&u=0 &&{\rm on }\ \partial \Omega,
		\end{aligned}
		\right.
	\end{equation*}
	where $f$ is, in general, a sign-changing function, 
	which has been attracting the attention of researchers in recent decades, see, e.g.,
	\cite{BPT1,DHI1,KQU1,KQU,QS}, and we also refer to 
	\cite{alama-tar,berest,drabek-poh,ilyasov2,IS} for the consideration of the superhomogeneous case $q>p$.

	Although the $q$-homogeneous terms in \eqref{eq:D} are always subordinate ``at infinity'' and dominant ``at zero'' to the $p$-homogeneous ones, a finer nature of their interplay crucially depends not only on the relation between $\alpha$, $\beta$ and eigenvalues of $-\Delta_p$, $-\Delta_q$, but also on the relation between $p$ and $2q$.
	More precisely, it is shown in \cite{BobkovTanaka2017} that the energy functional $\E$  has very different behavior in a neighborhood of the point $(\alpha,\beta)=(\lambda_1(p),\beta_*)$ depending on whether $p<2q$ or $p>2q$, see \eqref{eq:lambda1r} and \eqref{def:values} in Section~\ref{sec:notations} for the used notation.
	
	\medskip
	On one hand, the theory of existence and qualitative properties of \textit{sign-constant} solutions of the problem \eqref{eq:D} is relatively well developed in the contemporary literature, see, e.g., \cite{BobkovTanaka2015,BobkovTanaka2017,QS,T-2014} and the overview \cite{marcomasconi}, to mention a few, and also references in and to these works.
	On the other hand, the theory of existence of \textit{sign-changing} solutions of \eqref{eq:D} is  investigated in, e.g., \cite{BobkovTanakaNodal1,T-Uniq}, and we again refer to the bibliography surrounding these works.
	Several results from the above mentioned articles will be recalled in the following sections.
	
	The main aim of the present work is to develop the theory of existence of \textit{abstract} solutions of the problem \eqref{eq:D} by means of the Ljusternik--Schnirelmann theory. 
	By abstract solutions we understand those solutions whose pointwise properties, such as information on the sign, are  unknown. 
	In this direction, for any $p > q \geq 2$, $\rho>0$, and $\nu, \eta \geq 0$ with $\nu+\eta>0$ the general result of \cite{FN} (see also \cite[Section~44.6]{zeidler} and references therein) yields the existence of an infinite sequence of positive numbers (``eigenvalues'') $\lambda_k(\nu,\eta;\rho)$ such that for any $k$, \eqref{eq:D} with $\alpha=\lambda_k(\nu,\eta;\rho)\nu$ and $\beta=\lambda_k(\nu,\eta;\rho)\eta$ possesses a solution $u_k$ satisfying $\|\nabla u_k\|_p^p + \|\nabla u_k\|_q^q = \rho$, and, moreover, $\lambda_k(\nu,\eta;\rho) \to +\infty$ as $k \to +\infty$.
	We also refer to \cite{CS,tienari} for a more recent treatment of related problems in the framework of Orlicz--Sobolev spaces and for the discussion of some qualitative properties of the corresponding eigenvalues.
	The existence of abstract solutions of the problem \eqref{eq:D} with \textit{fixed} parameters $\alpha$, $\beta$ is investigated in \cite{col,ZongoRuf1,ZongoRuf2}.
	In \cite{col}, as a particular case of a more general result, the author obtains multiplicity of solutions when the parameters satisfy certain relations with respect to the so-called quasi-eigenvalues of the $p$- and $(p,q)$-Laplacians.
	When either $\alpha=0$ or $\beta=0$, multiplicity and bifurcation results when the parameter is nearby a variational eigenvalue of the $q$- or $p$-Laplacian, respectively, are studied in \cite{ZongoRuf1,ZongoRuf2}.
	
	Our aim consists in investigating the multiplicity of solutions of the problem \eqref{eq:D} with fixed parameters $\alpha$ and $\beta$ in ranges different from and extending those found in \cite{col,ZongoRuf1,ZongoRuf2}.
	For this purpose, we employ several minimax variational methods 
	(see, e.g., \cite{Clark,rabinowitz,Struwe,Szulkin}) and determine three generally different regions on the $(\alpha,\beta)$-plane where the problem possesses a given number of distinct pairs of solutions with a prescribed sign of energy.
	
	\smallskip
	The paper has the following structure. 
	In the rest of this section, we first set up notation and recall some known facts, and then state our main results.
	Section~\ref{sec:prelim0} contains preliminary material and auxiliary results.
	In Section~\ref{sec:proofs}, we provide proofs of the main results.
	Appendix~\ref{sec:appendixA} contains a discussion on alternative characterizations of variational eigenvalues of the $q$-Laplacian, some of which are used in the proofs of the main results.
	Finally, in Appendix~\ref{sec:appendixB}, we discuss further properties of eigenvalues which justify the nontriviality of assumptions of the first main theorem.

	\subsection{Some facts and notation}\label{sec:notations}
	
	Let $r>1$.
	In the standard way, if for some $\lambda \in \mathbb{R}$ there exists a nonzero function $\varphi \in \Wr$ such that
	$$
	\intO |\nabla \varphi|^{r-2} \nabla \varphi \nabla \xi \,dx
	=
	\intO |\varphi|^{r-2} \varphi \xi \, dx
	\quad \text{for any}~ \xi \in \Wr,
	$$
	then $\lambda$ and $\varphi$ are called an eigenvalue and eigenfunction of the $r$-Laplacian, respectively.
	We will use the notation $\sigma(-\Delta_r)$ for the set of all eigenvalues (i.e., the spectrum) of $-\Delta_r$, and denote by 
	$ES(r;\lambda)$ the closure of the set of all eigenfunctions (i.e., the eigenspace) of $-\Delta_r$ associated with $\lambda \in \mathbb{R}$. 
	Note that $ES(r;\lambda) = \emptyset$ if $\lambda$ does not belong to $\sigma(-\Delta_r)$.
	
	\begin{remark}\label{rem:reg}
		Thanks to the imposed $C^{1,\kappa}$-regularity of $\Omega$, any eigenfunction $\varphi$ of the $r$-Laplacian, as well as any weak solution $u$ of the problem~\eqref{eq:D}, belongs to $C_0^{1,\mu}(\overline{\Omega})$ for some $\mu \in (0,1)$.
		This regularity result follows from	\cite[Theorem~1]{Lieberman} (see a discussion in \cite[p.~320]{L} in the case of \eqref{eq:D}) by noting that $\varphi,u$ are bounded in $\Omega$, which can be shown using the standard bootstrap arguments.
	\end{remark}
	
	We will deal with a sequence of ``variational'' eigenvalues $\{\lambda_k(r)\}$ of $-\Delta_r$ 
	constructed via the Krasnoselskii genus as follows. 
	Denote
	\begin{equation}\label{eq:lambdak}
		\lambda_k(r) 
		=
		\inf
		\left\{
		\max_{u\in A} \frac{\|\nabla u\|_r^r}{\|u\|_r^r}:~ A\in \Sigma_k(r)
		\right\},
	\end{equation}
	where
	\begin{equation}\label{eq:Sigma-r}
		\Sigma_k(r) =\left\{\, A\subset W_0^{1,r} \setminus \{0\}:~
		A\ {\rm is\ symmetric,\ compact\ in}\ W^{1,r}_0,\ 
		{\rm and}\ \gamma(A)\ge k\,\right\}
	\end{equation}
	and $\gamma(A)$ is the Krasnoselskii genus defined as
	\begin{equation}\label{eq:genus}
		\gamma(A)
		=
		\inf \{ k\in\mathbb{N}:~
		\exists \,h \in C(A,\mathbb{R}^k\setminus\{0\}), ~
		h\ {\rm is\ odd}\}.
	\end{equation}
	We set $\gamma(A)=+\infty$ if a continuous odd mapping in the definition \eqref{eq:genus} does not exist, and we also set $\gamma(\emptyset)=0$. 
	It is known that every $\lambda_k(r)$ is an eigenvalue of $-\Delta_r$, 
	$\lambda_1(r)$ and $\lambda_2(r)$ are the actual first and second eigenvalues, respectively, and $\lambda_k(r) \to +\infty$ as $k \to +\infty$, see, e.g., \cite{cuesta} for an overview.
	
	In Appendix~\ref{sec:appA1}, we show that $\lambda_k(q)$ can be alternatively characterized by the smaller constraint set $\Sigma_k(r)$ with $r>q$, which will be used in the proofs of our main results (see Remark~\ref{rem:gap}). 
	In addition, Appendix~\ref{sec:appA2} contains a discussion about the characterization of $\lambda_k(q)$ via $\Sigma_k(r)$ with $r<q$, in which case the constraint set is larger.
	
	Let us mention explicitly that the first eigenvalue of the $r$-Laplacian can be defined as
	\begin{equation}\label{eq:lambda1r}
		\lambda_1(r) =
		\inf\left\{
		\frac{\|\nabla u\|_r^r}{\|u\|_r^r}:~ u \in W_0^{1,r} \setminus \{0\}
		\right\}.
	\end{equation}
	It is known that $\lambda_1(r)$ is attained, it is isolated and simple.
	Moreover, any corresponding minimizer (i.e., the first eigenfunction of $-\Delta_r$) has a strict sign in $\Omega$, while any higher eigenfunction of $-\Delta_r$ is necessary sign-changing, see \cite{cuesta} and references therein. 
	We will denote the first eigenfunction by $\varphi_r$ and assume, without loss of generality that $\varphi_r>0$ in $\Omega$ and $\|\varphi_r\|_r = 1$.
	
	\medskip
	Consider the family of critical parameters
	\begin{equation}\label{def:beta_*}
		\beta_*(\alpha) = \inf\left\{ \frac{\|\nabla u\|_q^q}{\|u\|_q^q}:~
		u\in\W\setminus \{0\} \text{ and } H_\alpha(u)\le 0\right\},
	\end{equation}
	and set $\beta_*(\alpha)=+\infty$ if $\alpha<\lambda_1(p)$. 
	Several main properties of $\beta_*(\alpha)$ are collected in \cite[Proposition~7]{BobkovTanaka2017}.
	In particular, $\beta_*(\alpha)$ is continuous for $\alpha>\lambda_1(p)$, 
	strictly decreasing for $\lambda_1(p) \leq \alpha \leq \alpha_*$, $\beta_*(\alpha)=\lambda_1(q)$ for $\alpha \geq \alpha_*$, and $\beta_*(\lambda_1(p)) = \beta_* > \lambda_1(q)$, where
	\begin{equation}\label{def:values} 
		\alpha_*=\frac{\|\nabla \varphi_q\|_p^p}{\|\varphi_q\|_p^p} 
		\quad \text{and} \quad 
		\beta_*
		=
		\frac{\|\nabla \varphi_p\|_q^q}{\|\varphi_p\|_q^q}.
	\end{equation}
	Let us explicitly note that $\alpha_*$ is well defined thanks to the $C^{1,\mu}(\overline{\Omega})$-regularity of $\varphi_q$, see Remark~\ref{rem:reg}.
	A few additional properties of $\beta_*(\alpha)$ needed for the present work are placed in Section~\ref{sec:prelim} and also in Appendix~\ref{sec:appendixB}.
	
	\medskip
	We will denote the (sub- and super-) level sets of a functional $I$ on $\W$ (with main examples of $I$ being $\E$, $\Ha$, $\Gb$) as
	\begin{align*}
		&[I\le c ]=
		\left\{u\in \W:~ I(u)\le c\right\}, 
		\quad 
		[I< c ]=
		\left\{u\in \W:~ I(u)< c\right\}, 
		\\
		&[I\ge c] = [-I \le -c],
		\quad
		[I>c]=[-I<-c],
		\quad 
		[I = c]=
		\left\{u\in \W:~ I(u)= c\right\}.
	\end{align*}
	Also, we will denote the critical set of the energy functional $\E$ and its restriction to a level $c \in \mathbb{R}$ as
	\begin{equation}\label{def:critical_set} 
		K=\{u\in\W:~ \E^\prime(u)=0\,\} 
		\quad {\rm and} \quad 
		K_c=K\cap [\E=c],
	\end{equation}
	respectively.

	%%%%%%%%%%%%%%%%%%%%%%%%%%%%%%%%%%%%%%%%%%%%%%%
	%%%%%%%%%%%%%%%%%%%%%%%%%%%%%%%%%%%%%%%%%%%%%%%%%
	\subsection{Main results}\label{sec:main}
	
	In our mains results -- Theorems~\ref{thm1}, \ref{thm2}, and \ref{thm3} -- we provide multiplicity of solutions for the problem \eqref{eq:D} in the sets $[\E<0]$ and $[\E>0]$, see Figure~\ref{fig:main}. 
	We start with the negative energy case.
	
	\begin{theorem}\label{thm1} 
		Let $k\in\mathbb{N}$ and either of the following assumptions be satisfied:
		\begin{enumerate}[label={\rm(\roman*)}]
			\item\label{thm1:1}  $\alpha < \lambda_1(p)$ and $\lambda_k(q) < \beta$;
			\item\label{thm1:2}  $\alpha = \lambda_1(p)$, and $\lambda_k(q) < \beta < \beta_*$;
			\item\label{thm1:2x}  $\alpha = \lambda_1(p)$, $\lambda_k(q) < \beta = \beta_*$, and $p>2q$;
			\item\label{thm1:3}  $\lambda_1(p)<\alpha < \alpha_*$ and $\lambda_k(q) < \beta \leq \beta_*(\alpha)$.
		\end{enumerate}
		Then $\E$ has at least $k$ distinct pairs of critical points in 
		$[\E<0]$, that is, \eqref{eq:D} has at least $k$ distinct pairs of 
		(nontrivial) solutions with negative energy. 
	\end{theorem} 
	
	The result of Theorem~\ref{thm1} is known when $k=1$, see \cite[Proposition~1 and Theorem~2.6, 2.7]{BobkovTanaka2017}.
	In addition, at least when $\alpha < \lambda_2(p)$, $\beta > \lambda_2(q)$, and $(\alpha,\beta) \neq (\lambda_1(p),\beta_*)$, \cite[Theorem~1.7]{BobkovTanakaNodal1} guarantees the existence of a sign-changing solution of \eqref{eq:D} with negative energy.
	We refer to Theorem~\ref{thm1-alph} below for a particular case of Theorem~\ref{thm1} formulated in different terms, see also Figure~\ref{fig:perturb}.
	
	\begin{remark}\label{rem:poinc}
		In the case $N \geq 2$, the $C^2$-regularity of $\Omega$ and, occasionally, the connectedness of $\partial\Omega$ imposed in \cite{BobkovTanaka2017} can be weakened to the $C^{1,\kappa}$-regularity with $\kappa \in (0,1)$.
		Indeed, these more restrictive assumptions were mainly required in order to employ in the proofs the improved Poincar\'e inequality from \cite{takac}, see \cite[Remark~5]{BobkovTanaka2017}.
		However, it was recently shown in \cite[Theorem~1.2]{BK} that the result of \cite{takac} remains valid assuming only that $\Omega$ is of class $C^{1,\kappa}$.
	\end{remark}
	
	Let us observe that for $N=1$, \cite[Lemma~A.2]{BobkovTanakaNodal1} gives $\beta_* < \lambda_2(q)$, and hence the assumptions \ref{thm1:2}, \ref{thm1:2x}, \ref{thm1:3} of Theorem~\ref{thm1} can be satisfied only in the case $k=1$, in which the existence is known. 
	In order to assure that the statement of  Theorem~\ref{thm1} under the assumptions \ref{thm1:2}, \ref{thm1:2x}, \ref{thm1:3} can be nontrivial for any given $k \geq 2$ in the higher dimensional case $N \geq 2$, in Lemma~\ref{lem:beta*>lambdak} below we construct a smooth bounded domain $\Omega$ for which $\lambda_k(q) < \beta_*$.

	\begin{figure}[!ht]
		\center{
			\includegraphics[scale=0.65]{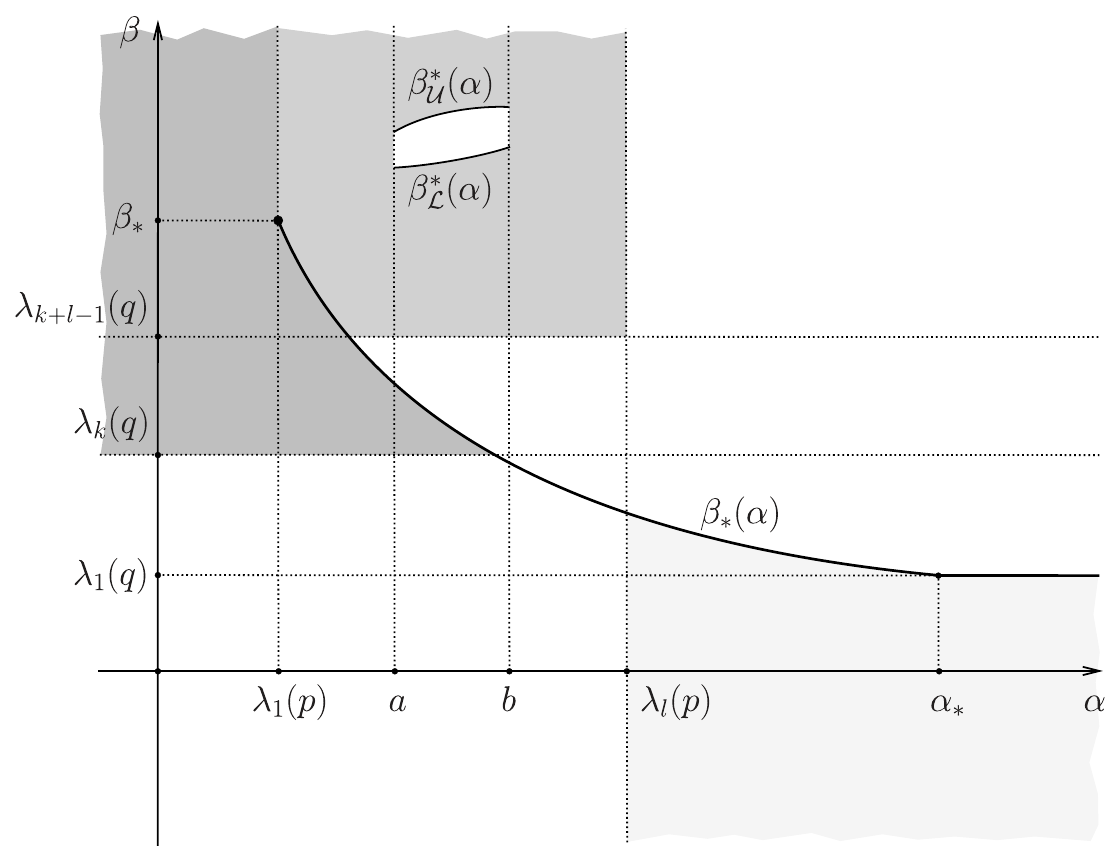}
		}
		\caption{A schematic plot of admissible assumptions of Theorem~\ref{thm1} (dark gray) assuming $\lambda_k(q)<\beta_*$, Theorem~\ref{thm2} (gray $+$ a part of dark gray) assuming $\lambda_{k+l-1}(q)<\beta_*$ and $[a,b] \subset \sigma(-\Delta_p)$, and Theorem~\ref{thm3} (light gray) assuming $\lambda_l(p)<\alpha_*$.}
		\label{fig:main}
	\end{figure}
	
	\medskip
	The multiplicity of solutions of \eqref{eq:D} with negative energy can also be obtained for another range of parameters $\alpha$ and $\beta$.
	\begin{theorem}\label{thm2} 
		Let $k, l \in \mathbb{N}$.
		Assume that $\alpha < \lambda_l(p)$ and 
		${\lambda}_{k+l-1}(q)<\beta$.	
		Assume, in addition, that 
		\begin{equation}\label{resonant:G}
			G_\beta(v)\not=0
			\quad 
			{\rm for\ all}\ v\in ES(p;\alpha)\setminus\{0\} 
		\end{equation}
		provided $\alpha\in\sigma(-\Delta_p)$. 
		Then $\E$ has at least $k$ distinct pairs of critical points in $[\E<0]$. 
	\end{theorem} 
	
	\begin{remark} 
		Consider the critical values
		\begin{align*}
			\beta^*_\mathcal{L}(\alpha)
			&=
			\inf
			\left\{
			\frac{\|\nabla u\|_q^q}{\|u\|_q^q}:~
			u
			\in
			ES(p;\alpha)\setminus\{0\}
			\right\},
			\\	
			\beta^*_\mathcal{U}(\alpha)
			&=
			\sup
			\left\{
			\frac{\|\nabla u\|_q^q}{\|u\|_q^q}:~
			u
			\in
			ES(p;\alpha)\setminus\{0\}
			\right\},
		\end{align*}
		and set $\beta^*_\mathcal{L}(\alpha) = +\infty$ and $\beta^*_\mathcal{U}(\alpha) = -\infty$ when $\alpha \not\in\sigma(-\Delta_p)$.
		It is shown in \cite[Lemma~3.6]{BobkovTanakaNodal1} that $\beta^*_\mathcal{U}(\alpha)<+\infty$ for any $\alpha \in \mathbb{R}$.
		It is clear that 
		the condition \eqref{resonant:G} of Theorem~\ref{thm2} holds 
		for any $\beta \in (-\infty,\beta^*_\mathcal{L}(\alpha)) \cup (\beta^*_\mathcal{U}(\alpha),+\infty)$ in the resonant case $\alpha\in\sigma(-\Delta_p)$. 
	\end{remark} 
	
	The existence of at least one sign-changing solution of \eqref{eq:D} with negative energy for $\alpha$, $\beta$ satisfying similar assumptions to that of Theorem~\ref{thm2} is given by \cite[Theorem~1.6]{BobkovTanakaNodal1}.
	
	Notice that if $\alpha>\lambda_1(p)$ and $\beta$ is large enough, then \textit{any} solution provided by Theorem~\ref{thm2} is sign-changing. 
	Indeed, \cite[Theorem~2.2]{BobkovTanaka2015} postulates the existence of a curve $\mathcal{C}$ in the quadrant $[\lambda_1(p),+\infty) \times[\lambda_1(q),+\infty)$ of the $(\alpha,\beta)$-plane which separated ranges of the existence and nonexistence of positive (in fact, nontrivial and sign-constant) solutions of \eqref{eq:D} depending on whether $(\alpha,\beta)$ lies below or above $\mathcal{C}$, respectively. 
	Thus, if $(\alpha,\beta)$ lies above $\mathcal{C}$ and the assumptions of Theorem~\ref{thm2} are satisfied, then the obtained solutions change sign in $\Omega$.
	
	It is not hard to see that Theorems~\ref{thm1} and \ref{thm2} are equivalent whenever $\alpha<\lambda_1(p)$.
	On the other hand, if $\alpha \geq \lambda_1(p)$, then the ranges for $\beta$ provided by these theorems for a given $k$  are already different. 
	In particular, the following result is a direct corollary of Theorems~\ref{thm1} and \ref{thm2}.
	\begin{corollary}\label{cor:thm12}
		Let $k \in \mathbb{N}$.
		Assume that $\alpha=\lambda_1(p)$ and either of following assumptions is satisfied:
		\begin{enumerate}[label={\rm(\roman*)}]
			\item\label{cor:thm12:1}  $\lambda_k(q) < \beta < \beta_*$;
			\item\label{cor:thm12:3}  $\lambda_k(q) < \beta = \beta_*$ and $p>2q$;
			\item\label{cor:thm12:4}  $\lambda_{k+1}(q) < \beta \neq \beta_*$.
		\end{enumerate}
		Then $\E$ has at least $k$ distinct pairs of critical points in 
		$[\E<0]$.
	\end{corollary}
	It would be interesting to know whether the existence of $k$ distinct pairs of solutions of \eqref{eq:D} with $\alpha=\lambda_1(p)$ takes place for any $\beta>\lambda_k(q)$, $\beta \neq \beta_*$.
	
	\medskip
	Let us now discuss the multiplicity of solutions with positive energy.
	\begin{theorem}\label{thm3} 
		Let $l \in\mathbb{N}$. 
		Assume that $\alpha>\lambda_l(p)$ and $\beta < \beta_*(\alpha)$.
		Then $\E$ has at least $l$ distinct pairs of critical points in 
		$[\E>0]$, that is, \eqref{eq:D} has at least $l$ distinct pairs  of (nontrivial)
		solutions with positive energy. 
	\end{theorem} 
	
	Similarly to the discussion after Theorem~\ref{thm1}, the result of Theorem~\ref{thm3} is known when $l=1$, see, e.g., \cite[Theorems~2.3 and 2.7]{BobkovTanaka2017}. 
	Moreover, at least for $\alpha>\lambda_2(p)$ and $\beta < \lambda_1(p)$, the existence of a sign-changing solution with positive energy is given by \cite[Theorem~1.5]{BobkovTanakaNodal1}.
	
	Note that, in contrast to Theorem~\ref{thm1}, 
	the point $(\alpha, \beta)$ satisfying the assumptions of Theorem~\ref{thm3} with $l \geq 2$ might be located above the quadrant $(\lambda_1(p),+\infty)\times(-\infty,\lambda_1(q))$ already in the case $N=1$. 
	Indeed, in the one-dimensional case, for any $l \geq 2$ there exist $q_0$, $p_0$ with $1<q_0<p_0$ such that $\lambda_l(p) < \alpha_*$ provided $1<q<q_0$ and $p>p_0$, see \cite[Lemma~A.3]{BobkovTanaka2017}.
	Consequently, in view of the properties of $\beta_*(\alpha)$ (see Section~\ref{sec:notations}), there exist $\alpha$, $\beta$ such that $\lambda_l(p) < \alpha < \alpha_*$ and $\lambda_1(q) < \beta < \beta_*(\alpha)$.
	We also refer to Remark~\ref{rem:astar} below for the case $N \geq 2$.
	
	\medskip
	Combining Theorems~\ref{thm1}~\ref{thm1:3} and \ref{thm3}, we have the following result which extends \cite[Theorem~2.7]{BobkovTanaka2017}.
	\begin{corollary} 
		Let $k,l \in \mathbb{N}$.
		Assume that $\lambda_l(p) < \alpha < \alpha_*$ and  $\lambda_k(q)<\beta < \beta_*(\alpha)$.
		Then \eqref{eq:D} has at least $k+l$ distinct pairs of 
		nontrivial solutions. 
	\end{corollary} 
	
	We schematically depict the results of Theorems~\ref{thm1}, \ref{thm2}, \ref{thm3} on Figure~\ref{fig:main}.

	%%%%%%%%%%%%%%%%%%%%%%%%%%%%%%%%%%
	%%%%%%%%%%%%%%%%%%%%%%%%%%%%%%%%%%%%
	
	\section{Preliminaries}\label{sec:prelim0}
	In this section, we collect several auxiliary results needed for the proofs of the main theorems.
	
	\subsection{Properties of \texorpdfstring{$\beta_*(\alpha)$}{beta*(alpha)}}\label{sec:prelim}
	Let us provide several properties of the mapping $\beta_*(\alpha)$ defined in \eqref{def:beta_*}, in addition to those stated in \cite[Proposition~7]{BobkovTanaka2017}.
	For visual convenience, we note that  $\beta_*(\alpha)$ can be equivalently defined as
	\begin{equation}\label{def:beta_*2}
		\beta_*(\alpha) = \inf\left\{ \frac{\|\nabla u\|_q^q}{\|u\|_q^q}:~
		u\in\W\setminus \{0\} \text{ and } \frac{\|\nabla u\|_p^p}{\|u\|_p^p} \le \alpha \right\}. 
	\end{equation}
	
	\begin{lemma}\label{lem:HG<0} 
		Let $\alpha\ge \lambda_1(p)$.
		Then the following assertions hold:
		\begin{enumerate}[label={\rm(\roman*)}]
			\item\label{lem:H<0} 
			If $\beta<\beta_*(\alpha)$, then 
			$[H_\alpha\le 0]\subset \{0\}\cup[G_\beta>0]$, whence 
			$[H_\alpha=0] \cap [\E<0] =\emptyset$. 
			\item\label{lem:H<02} 
			If $\beta \leq \beta_*(\alpha)$, then 
			$[G_\beta < 0]\subset [H_\alpha > 0]$.
			\item\label{lem:H<03} 
			If $\lambda_1(q) < \beta \leq \beta_*(\alpha)$, then 
			$[G_\beta \leq 0]\subset [H_\alpha \geq 0]$, whence $[\Gb = 0] \cap [\E<0] = \emptyset$.
		\end{enumerate}
	\end{lemma} 
	\begin{proof}
		The assertions \ref{lem:H<0} and \ref{lem:H<02}  directly follow from the definition of $\beta_*(\alpha)$.
		Let us prove the assertion \ref{lem:H<03}. 
		In view of \ref{lem:H<02}, it is sufficient to show that $[G_\beta = 0]\subset [H_\alpha \geq 0]$.
		Suppose, by contradiction, that there exists $u \in \W$ such that $G_\beta(u) = 0$ and $H_\alpha(u) < 0$. 
		Clearly, this yields $\beta=\beta_*(\alpha)$ by the definition~\eqref{def:beta_*} of $\beta_*(\alpha)$.
		Since $\Gb$ and $\Ha$ are even, we may assume that $u \geq 0$ in $\Omega$.
		That is, $u$ is a nonnegative minimizer of $\beta_*(\alpha)$. 
		Since $[\Ha<0]$ is an open neighborhood of $u$, 
		the inequality $\Ha(u)<0$ means that $u$ is an interior point of the admissible set of $\beta_*(\alpha)$.
		This implies that $\langle\Gb'(u),\xi\rangle=0$ for any $\xi \in C_0^{\infty}(\Omega)$, 
		and hence $u$ is a nonnegative eigenfunction of the $q$-Laplacian.
		However, it contradicts the assumption $\lambda_1(q)<\beta$ and the fact that the only sign-constant eigenfunction is the first one, see Section~\ref{sec:notations}.
	\end{proof}

	Let us introduce a related family of critical points: 
	\begin{equation}\label{def:alpha_*}
		\alpha_*(\beta) = \inf\left\{ \frac{\|\nabla u\|_p^p}{\|u\|_p^p}:~
		u\in\W\setminus \{0\} \text{ and } G_\beta(u)\le 0\right\}, 
	\end{equation}  
	and set $\alpha_*(\beta) = +\infty$ if $\beta<\lambda_1(q)$. 
	Note that $\alpha_*(\beta)$ can be equivalently defined as
	\begin{equation}\label{def:alpha_*2}
		\alpha_*(\beta) = \inf\left\{ \frac{\|\nabla u\|_p^p}{\|u\|_p^p}:~\,
		u\in\W\setminus \{0\} \text{ and } \frac{\|\nabla u\|_q^q}{\|u\|_q^q} \leq \beta \right\}.
	\end{equation} 
	It is evident that $\alpha_*(\beta)\ge \lambda_1(p)$ for all $\beta\ge \lambda_1(q)$.
	Since the first eigenfunction $\varphi_p$ of $-\Delta_p$ is an admissible function for the definition \eqref{def:alpha_*} whenever $\beta \ge \beta_*$, we have $\alpha_*(\beta)=\lambda_1(p)$ for any $\beta\ge \beta_*$. 
	In the following lemma, we show that 
	$\alpha_*(\beta)$ and $\beta_*(\alpha)$ describe the same curve in the quadrant $(\lambda_1(p),+\infty)\times (\lambda_1(q),+\infty)$ of the $(\alpha,\beta)$-plane.
	As a consequence, all the properties of $\beta_*(\alpha)$ in this quadrant are directly transferred to that of $\alpha_*(\beta)$.
	\begin{lemma}\label{lem:curvebelow} 
		Let $(\alpha,\beta)\in\mathbb{R}^2$. 
		Then the following cases are equivalent: 
		\begin{enumerate}[label={\rm(\roman*)}] 
			\item\label{lem:curvebelow:1}  $\beta>\lambda_1(q)$ 
			and $\lambda_1(p)< \alpha \leq \alpha_*(\beta)$; 
			\item\label{lem:curvebelow:2}  $\alpha>\lambda_1(p)$ 
			and $\lambda_1(q)< \beta \leq \beta_*(\alpha)$. 
		\end{enumerate}
	\end{lemma}
	\begin{proof} 
		Suppose first that under the condition \ref{lem:curvebelow:1} there exists some pair 
		$(\alpha,\beta)$ violating \ref{lem:curvebelow:2},
		that is, $\beta>\beta_*(\alpha)$.
		By the definition~\eqref{def:beta_*} of $\beta_*(\alpha)$, this inequality yields the existence of $u \in \W$ such that $H_\alpha(u) \leq 0$ and $G_\beta(u) < 0$.
		Consequently, $u$ is an admissible function for the definition \eqref{def:alpha_*} of $\alpha_*(\beta)$, which gives
		$$
		\frac{\|\nabla u\|_p^p}{\|u\|_p^p}
		\leq
		\alpha 
		\leq 
		\alpha_*(\beta)
		\leq
		\frac{\|\nabla u\|_p^p}{\|u\|_p^p},
		$$
		i.e., $u$ is a minimizer of $\alpha=\alpha_*(\beta)$.
		Thanks to the evenness of $H_\alpha$ and $G_\beta$, we can assume that $u$ is nonnegative in $\Omega$, without loss of generality.
		Since $[\Gb<0]$ is an open neighborhood of $u$, 
		$u$ is an interior point of the admissible set of $\alpha_*(\beta)$.
		We conclude that $H_\alpha'(u)=0$ in $(\W)^*$, that is, $\alpha \in \sigma(-\Delta_p)$.
		Recalling that $u \geq 0$, we have $\alpha = \lambda_1(p)$, which contradicts our assumption $\alpha > \lambda_1(p)$.
		Thus, we have shown that \ref{lem:curvebelow:1} implies~\ref{lem:curvebelow:2}.
		
		Arguing in much the same way as above, one can justify that \ref{lem:curvebelow:2} implies \ref{lem:curvebelow:1}, which completes the proof of their equivalence.
	\end{proof} 
	
	In view of Lemma~\ref{lem:curvebelow}, Theorem~\ref{thm1} can be reformulated in terms of $\alpha_*(\beta)$, which might be instructive from the point of view of the bifurcation theory, see Figure~\ref{fig:perturb}. 
	In particular, a simplified version of Theorem~\ref{thm1} can be stated as follows (cf.\ Theorem~\ref{thm3}).
	\begin{theorem}\label{thm1-alph} 
		Let $k\in\mathbb{N}$. Assume that $\beta>\lambda_k(q)$ and $\alpha < \alpha_*(\beta)$.
		Then $\E$ has at least $k$ distinct pairs of critical points in 
		$[\E<0]$. 
	\end{theorem}

	\begin{figure}[!ht]
		\center{
			\includegraphics[scale=0.65]{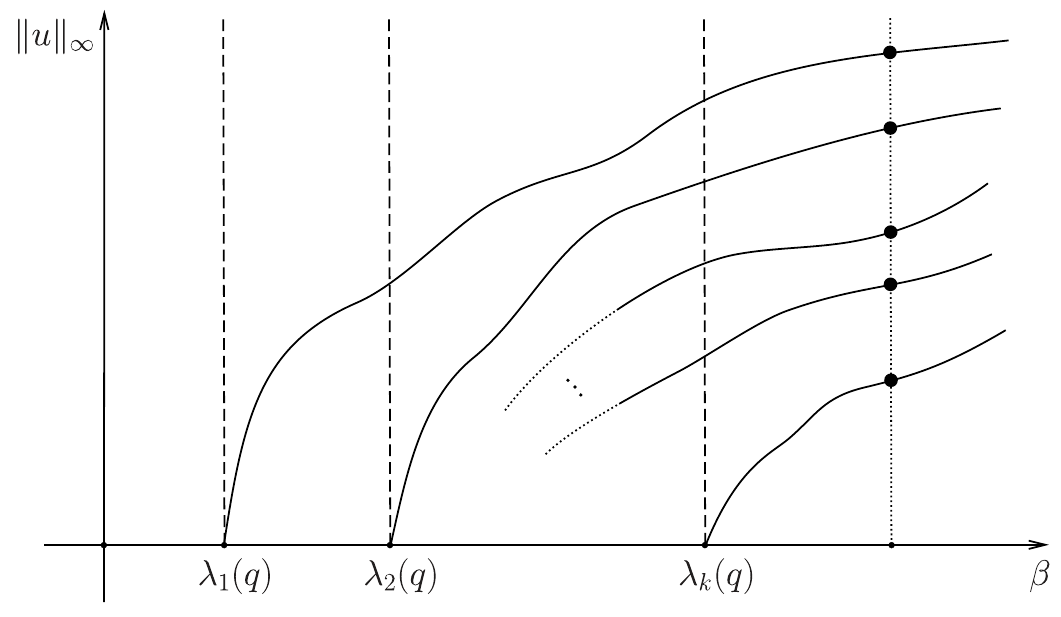}
		}
		\caption{A schematic plot of branches of solutions to the (perturbed) problem \eqref{eq:D} emanating from eigenvalues of the (unperturbed) $q$-Laplacian assuming $\alpha < \alpha_*(\beta)$.}
		\label{fig:perturb}
	\end{figure}

	\subsection{Properties of the Nehari manifold}\label{sec:nehari}
	
	Let us introduce the Nehari manifold for the problem \eqref{eq:D} in the standard way:
	\begin{equation*}\label{def:Nehari} 
		\N =\{ v\in\W\setminus\{0\}:~
		\langle \E^\prime(v),v \rangle=H_\alpha(v)+G_\beta(v)=0\,\}. 
	\end{equation*}
	Defining a functional $\F \in C^1(\W,\mathbb{R})$ as 
	\begin{equation}\label{def:F} 
		\F(u) = H_\alpha (u)+G_\beta(u), 
	\end{equation}
	we have $\N=[\F=0]\setminus \{0\}$, in the notation of Section~\ref{sec:notations}.
	Moreover, we will denote by $T_u \N$ the tangent space of $\N$ at $u \in \N$, and by $\EN$ the restriction of $\E$ to $\N$, that is, $\EN =\E\big|_{\N}$. 
	
	Let us provide a few properties of $\N$, some of which are basic and known. 
	Primarily, let us observe that if $u \in \N$, then
	\begin{equation}\label{eq:EHG}
		\E(u) 
		= 
		-\frac{p-q}{pq} \Ha(u)
		=
		\frac{p-q}{pq} \Gb(u).
	\end{equation}
	\begin{proposition}[\protect{\cite[Proposition~6]{BobkovTanaka2015}}]\label{prop:minpoint}
		Let $u \in \W$. If $H_\alpha(u) \cdot G_\beta(u) < 0$, then there exists a unique extrema point 
		\begin{equation}\label{tu}
			t(u) = \left(\frac{-G_\beta(u)}{H_\alpha(u)}\right)^{\frac{1}{p-q}} = 
			\frac{|G_\beta(u)|^{\frac{1}{p-q}}}{|H_\alpha(u)|^{\frac{1}{p-q}}}
		\end{equation}
		of $\E(t u)$ with respect to $t > 0$, and $t(u) u \in \N$.
		More precisely, if $\Ha(u) < 0 < \Gb(u)$, 
		then $\E(t(u) u) >0$ and $t(u)$ is a unique maximum point of $t \mapsto \E(t u)$ on $(0,+\infty)$, while if $\Ha(u) > 0 > \Gb(u)$, 
		then $\E(t(u) u) <0$ and $t(u)$ is a unique minimum point of $t \mapsto \E(t u)$ on $(0,+\infty)$.
		Moreover,
		\begin{equation}\label{eq:J}
			\J(u) 
			:= 
			\E(t(u) u) 
			= -
			\mathrm{sign}(H_\alpha(u))\,
			\frac{p-q}{pq}\, \frac{|G_\beta(u)|^{\frac{p}{p-q}}}{|H_\alpha(u)|^{\frac{q}{p-q}}},
		\end{equation}
		where the so-called fibered functional $\J$ is $0$-homogeneous.
	\end{proposition}

	%%%%%%%%%%%%%%%%%%%%%%%%%%%%%%%%%%%%%%%%%%%%%%%%%%
	\begin{proposition}[{\cite[Lemma 2]{BobkovTanaka2015}}]\label{prop:critical} 
		Assume that $u\in\N$ satisfies $\E(u)\not=0$ (or, equivalently, $\Ha(u) \neq 0$ or $\Gb(u) \neq 0$, see \eqref{eq:EHG}). 
		If $u$ is a critical point of $\EN=\E\big|_{\N}$ (over $\N$), 
		then $u$ is a critical point of $\E$ (over $\W$). 
	\end{proposition} 
	
	Let us consider the following set:
	\begin{equation}\label{def:B^+} 
		B_{\alpha,\beta}^+ = \{v \in \W:\, H_\alpha(v)<0< G_\beta(v)\}
		= [H_\alpha<0] \cap [G_\beta>0]. 
	\end{equation}
	\begin{proposition}[\cite{BobkovTanaka2017}]\label{prop:LeastEnergy}
		Let $\alpha>\lambda_1(p)$ and $\beta<\beta_*(\alpha)$. 
		Then there exists $v\in \N \cap B_{\alpha,\beta}^+$ such that 
		$$
		\E(v) = \inf \{\E (u):~ u\in \N\cap B_{\alpha,\beta}^+\,\}
		>0.
		$$
		Moreover, $v$ is a positive solution of \eqref{eq:D}.
	\end{proposition} 
	\begin{proof}
		The case $\lambda_1(p) < \alpha <\alpha_*$ and $\beta<\beta_*(\alpha)$ is covered by \cite[Proposition~12]{BobkovTanaka2017}.
		The remaining case $\alpha \geq \alpha_*$ and $\beta<\beta_*(\alpha) \, (=\lambda_1(q))$ is given by \cite[Theorem~2.3]{BobkovTanaka2017}.
	\end{proof}

	\begin{proposition}\label{prop:MFD} 
		Let $\alpha>\lambda_1(p)$ and $\beta<\beta_*(\alpha)$. 
		Then 
		$\N\cap B_{\alpha,\beta}^+$ is a closed symmetric 
		$C^1$-submanifold of $\W$ of codimension~$1$, and 
		the tangent space $T_u \N$ at any $u\in \N\cap B_{\alpha,\beta}^+$ can be characterized as
		\begin{equation}\label{def:tangent_space} 
			T_u \N
			=
			\{
			v\in \W:~ \langle \F^\prime (u),v\rangle=0
			\}.
		\end{equation}
	\end{proposition}
	\begin{proof} 
		It is evident that $\N\cap B_{\alpha,\beta}^+$ is symmetric since 
		the involved functionals $\F$, $\Ha$,  $\Gb$ are even. 
		Setting
		\begin{equation}\label{def:d}
			\delta
			=
			\inf \{G_\beta (u):\,u\in \N\cap B_{\alpha,\beta}^+\,\} 
			=
			\frac{pq}{p-q}
			\inf \{\E (u):\,u\in \N\cap B_{\alpha,\beta}^+\,\},
		\end{equation}
		we get from Proposition~\ref{prop:LeastEnergy} that $\delta>0$, and hence
		$$
		\N\cap B_{\alpha,\beta}^+ 
		=[\F=0] \cap [G_\beta\ge \delta].
		$$
		Consequently, the set $\N\cap B_{\alpha,\beta}^+$ is closed. 
		Finally, any $u\in [\F=0] \cap [G_\beta \geq \delta]$ 
		satisfies 
		$$
		\langle \F^\prime(u),u\rangle 
		=pH_\alpha(u)+qG_\beta (u)
		=-(p-q)G_\beta (u)
		\leq -(p-q) \delta
		<0, 
		$$
		whence $\F^\prime(u)\not=0$ in $(\W)^*$.
		Therefore, the remaining assertions of the proposition follow from 
		the inverse function theorem, see, e.g., \cite[Theorem~43.C]{zeidler}.
		(In order to apply \cite[Theorem~43.C]{zeidler}, let us note that a continuous projection operator $P$ from $\W$ to the null space of $\F'(u)$ for a given $u \in \N\cap B_{\alpha,\beta}^+$ required in \cite[Definition~43.15]{zeidler} can be defined as $P(\xi) = \xi - u \, \langle\F'(u),\xi\rangle/\langle\F'(u),u\rangle$ for any $\xi \in \W$.)
	\end{proof} 
	
	In the standard way, the norm of the derivative of the restricted functional $\EN$ at $u \in \N\cap B_{\alpha,\beta}^+$ is defined by 
	\begin{equation}\label{def:der_EN} 
		\|\EN'(u)\|_{\N^*}
		=
		\sup\{\langle \E^\prime (u), v\rangle:~ v\in T_u\N, 
		\ \|\nabla v\|_p=1\,\}.
	\end{equation}
	In view of \eqref{def:tangent_space}, the duality lemma (see, e.g., \cite[Proposition~3.54]{perera}) implies that 
	$\|\EN'(u)\|_{\N^*}$ can be expressed as 
	\begin{equation}\label{def:der_EN_2} 
		\|\EN'(u)\|_{\N^*}
		=
		\min\{\|\E^\prime (u)-\lambda \F^\prime(u)\|_* 
		:~ \lambda\in\mathbb{R}\}\quad 
		{\rm for}\ u\in \N\cap B_{\alpha,\beta}^+, 
	\end{equation}
	where $\|\cdot\|_*$ is the standard operator norm, and we observe that the minimum is attained.
	
	Let us provide a compactness result for $\EN$ which will be needed for the proof of Theorem~\ref{thm3}.
	In Section~\ref{sec:propertiesE} below, we also discuss several related compactness results for the original energy functional $\E$.
	\begin{lemma}\label{lem:PS_positive} 
		Let $\alpha>\lambda_1(p)$ and 
		$\beta<\beta_*(\alpha)$. 
		Then $\EN$ satisfies the Palais--Smale condition 
		at any level $c>0$. 
	\end{lemma} 
	\begin{proof} 
		Let $\{u_n\}\subset \N$ be a Palais--Smale sequence for $\EN$ 
		at  $c>0$, that is, 
		$\EN(u_n)\to c$ and $\|\EN^\prime(u_n)\|_{\N^*} \to 0$ as $n \to +\infty$. 
		In view of \eqref{eq:EHG}, we have
		\begin{equation}\label{eq:PS+:1}  
			0<c+o(1)=\EN(u_n)=\E(u_n)
			=-\dfrac{p-q}{pq}H_\alpha (u_n)
			= \dfrac{p-q}{pq}G_\beta (u_n) 
		\end{equation}
		for any (sufficiently large) $n$.
		Thus, 
		$u_n\in[G_\beta>0]\cap [H_\alpha<0]$ and, consequently,  
		$u_n \in \N \cap B_{\alpha,\beta}^+$.
		Let us show the boundedness of $\{u_n\}$ in $\W$.
		Suppose, by contradiction, that $\|\nabla u_n\|_p\to +\infty$ as $n\to+\infty$, up to a subsequence.
		Considering the normalized functions $v_n = u_n/\|\nabla u_n\|_p$, we deduce that $\{v_n\}$ converges to some $v_0$ weakly in $\W$ and strongly in $L^p(\Omega)$, up to a subsequence. 
		Since $u_n\in\N$ and $\alpha>0$, we have
		$$
		0+o(1)=-\dfrac{G_\beta(v_n)}{\|\nabla u_n\|_p^{p-q}} 
		=H_\alpha (v_n)=1-\alpha\|v_n\|_p^p 
		\quad \text{as} ~n\to +\infty,
		$$
		which implies that $v_0 \neq 0$ in $\Omega$. Moreover, $v_0\in[H_\alpha\le 0]$ because $H_\alpha$ is weakly lower semicontinuous. 
		Therefore, we obtain from Lemma~\ref{lem:HG<0}~\ref{lem:H<0} that $v_0\in [G_\beta>0]$. 
		On the other hand, \eqref{eq:PS+:1} leads to 
		$$
		G_\beta (v_0)\le \liminf_{n\to+\infty} G_\beta (v_n) 
		=\liminf_{n\to+\infty} \dfrac{pq\,\E(u_n)}{\,(p-q)\|\nabla u_n\|_p^q\,} 
		= 0, 
		$$
		which is impossible.
		Thus, $\{u_n\}$ is bounded in $\W$.	
		
		By the expression \eqref{def:der_EN_2}, for each $n$ there exists  $\lambda_n\in\mathbb{R}$ such that  
		\begin{align}
			\|\EN^\prime(u_n)\|_{\N^*}
			&=\|\E^\prime(u_n)-\lambda_n \F^\prime (u_n)\|_*
			\label{eq:PS+:2} \\
			&=\| (1-p\lambda_n)H_\alpha^\prime(u_n)/p 
			+(1-q\lambda_n)G_\beta^\prime(u_n)/q\|_*. 
			\nonumber 
		\end{align} 
		Recalling that $u_n \in \N$ and $\|\EN^\prime(u_n)\|_{\N^*} \to 0$, we get
		\begin{align} 
			o(1)\,\|\nabla u_n\|_p 
			& =\langle \E^\prime(u_n)-\lambda_n \F^\prime (u_n), u_n \rangle 
			\nonumber \\ 
			& = (1-p\lambda_n)H_\alpha(u_n)+(1-q\lambda_n)G_\beta (u_n) 
			\nonumber \\ 
			& =-\lambda_n(p-q) H_\alpha (u_n) 
			=\lambda_n(p-q) G_\beta(u_n) 
			\label{eq:PS+:3} 
			\quad \text{as}~ n\to+\infty.~~~
		\end{align}
		Since $\{u_n\}$ is bounded in $\W$, we conclude from \eqref{eq:PS+:1} and \eqref{eq:PS+:3} that 
		$\lambda_n \to 0$. 
		Consequently, in view of \eqref{eq:PS+:2} and the fact that the boundedness of $\{u_n\}$ implies the boundedness of $\|\F'(u_n)\|_*$, we deduce that $\{u_n\}$ is a bounded Palais--Smale sequence of $\E$.
		Hence, the $(S_+)$-property of $-\Delta_p$ (see, e.g., \cite[Theorem 10]{dinica}) ensures that 
		$\{u_n\}$ has a subsequence which converges strongly in $\W$. 
		Since $\N \cap B_{\alpha,\beta}^+$ is closed by Proposition~\ref{prop:MFD}, 
		the limit of the convergent subsequence of $\{u_n\}$ 
		belongs to $\N \cap B_{\alpha,\beta}^+$, which completes the proof. 
	\end{proof}

	%%%%%%%%%%%%%%%%%%%%%%%%%%%%%%%%%%%%%%%%%%%%%%%
	%%%%%%%%%%%%%%%%%%%%%%%%%%%%%%%%%%%%%%%%%%%%%%%%
	%%%%%%%%%%%%%%%%%%%%%%%%%%%%%%%%%%%%%%%%%%%%%%%%%
	\subsection{Properties of the energy functional}\label{sec:propertiesE}
	
	Let us collect a few auxiliary results about properties of the energy functional $\E$.
	\begin{lemma}\label{lem:bdd_below} 
		Let either of the following assumptions be satisfied:
		\begin{enumerate}[label={\rm(\roman*)}]
			\item\label{lem:bdd_below:1} $\alpha < \lambda_1(p)$ and $\lambda_1(q) < \beta$;
			\item\label{lem:bdd_below:2} $\alpha = \lambda_1(p)$ and $\lambda_1(q) < \beta < \beta_*$;
			\item\label{lem:bdd_below:2x} $\alpha = \lambda_1(p)$, $\beta = \beta_*$, and $p \geq 2q$;
			\item\label{lem:bdd_below:3} $\lambda_1(p) < \alpha < \alpha_*$ and $\lambda_1(q) < \beta \leq \beta_*(\alpha)$.
		\end{enumerate}
		Then $-\infty < \inf_{[G_\beta<0]} \E < 0$.
	\end{lemma}
	\begin{proof} 
		Since $\lambda_1(q)<\beta$ in every case, we have $\varphi_q\in [G_\beta<0]$, and so 
		$\inf_{[G_\beta<0]} \E\le \E(t\varphi_q)<0$ for any sufficiently small $t>0$.
		Therefore, it remains to discuss only the boundedness of $\inf_{[G_\beta<0]} \E$ from below.
		In the cases \ref{lem:bdd_below:1} and \ref{lem:bdd_below:2}, we have $\inf_{\W} \E > -\infty$, see \cite[Propositions~1~(ii) and 3~(ii)]{BobkovTanaka2017}, respectively.
		In the case \ref{lem:bdd_below:2x}, we also have 
		$\inf_{\W} \E > -\infty$.
		This result is given by \cite[Propositions~3~(iii)]{BobkovTanaka2017}, and we refer to Remark~\ref{rem:poinc} for the discussion about sufficiency of the $C^{1,\kappa}$-regularity of $\Omega$.
		
		Finally, consider the case \ref{lem:bdd_below:3}.
		By Lemma~\ref{lem:HG<0}~\ref{lem:H<02}, any $u \in [\Gb < 0]$ satisfies $u \in [\Ha > 0]$.
		Consequently, there exists a unique minimum point $t(u)>0$ of $\E(tu)$ with respect to $t>0$, and $t(u)u \in \N$, see Proposition~\ref{prop:minpoint}.
		Therefore, using \cite[Theorems~2.5~(i) and~2.6]{BobkovTanaka2017}, we have
		$$
		\inf_{[G_\beta<0]} \E
		=
		\inf_{\N} \E > 
		-\infty,
		$$
		which completes the proof.
	\end{proof}

	\begin{lemma}\label{lem:surface} 
		Let $k \in \mathbb{N}$ and $\alpha \in \mathbb{R}$.
		If ${\lambda}_k(q)<\beta$, then 
		there exists 
		$A\in {\Sigma}_k(p)$ such that $A \subset [G_\beta<0] \cap [\E<0]$. 
	\end{lemma} 
	\begin{proof} 
		Thanks to Lemma~\ref{lem:la=la} (with $r=p$) and since ${\lambda}_k(q)<\beta$,
		we can find $\widetilde{A}\in {\Sigma}_k(p)$ satisfying 
		\begin{equation}\label{eq:lk<sup}
			\lambda_k(q) < \sup_{u\in \widetilde{A}} \frac{\|\nabla u\|_q^q}{\|u\|_q^q} < \beta.
		\end{equation}
		Notice that the supremum in \eqref{eq:lk<sup} is attained since $\widetilde{A}$ is compact in $\W$ and does not contain $0$.
		Consequently, 
		\begin{equation}\label{eq:g<01}
			\|\nabla v\|_q^q <\beta \|v\|_q^q 
			\quad \text{for any}~ v \in \widetilde{A}, 
		\end{equation}
		which yields $\widetilde{A} \subset [G_\beta<0]$. 
		Therefore, recalling that $q<p$, we can find $\tilde{t}>0$ such that 
		$$
		\E(tv)\le \dfrac{t^p}{p} \max_{u\in \widetilde{A}} H_\alpha(u) 
		+\dfrac{t^q}{q} \max_{u\in \widetilde{A}} G_\beta(u)<0
		$$
		for any $t \in (0,\tilde{t})$ and $v\in \widetilde{A}$.
		Fixing some $t \in (0,\tilde{t})$, we denote $A = t\widetilde{A}$. 
		It is evident that \eqref{eq:g<01} remains valid for any $v \in A$. 
		Moreover, $A$ is homeomorphic to $\widetilde{A}$ 
		by an odd continuous mapping, and hence $\gamma(A)=\gamma(\widetilde{A}) \geq k$
		(see, e.g., Lemma~\ref{lem:genus}). 
		That is, $A$ is the desired element of ${\Sigma}_k(p)$ which is a subset of $[G_\beta<0] \cap [\E<0]$.
	\end{proof}

	Let us now provide three compactness results.
	\begin{lemma}\label{lem:PS_negative2} 
		Let either of the following assumptions be satisfied:
		\begin{enumerate}[label={\rm(\roman*)}]
			\item\label{lem:PS_negative2:1}
			$\alpha\not\in\sigma(-\Delta_p)$;
			\item\label{lem:PS_negative2:2} 
			$\alpha\in\sigma(-\Delta_p)$ and 
			$G_\beta(v)\not=0$ for all $v\in ES(p;\alpha)\setminus\{0\}$. 
		\end{enumerate}
		Then $\E$ satisfies the Palais--Smale condition at any level $c\in\mathbb{R}$. 
	\end{lemma} 
	\begin{proof} 
		Let $\{u_n\}$ be any Palais--Smale sequence for $\E$ 
		at  $c\in\mathbb{R}$, that is, 
		$\E(u_n)\to c$ and $\|\E^\prime(u_n)\|_* \to 0$ as $n \to +\infty$. 
		Due to the $(S_+)$-property of the $p$-Laplacian (see, e.g., \cite[Theorem 10]{dinica}), $\E$ satisfies the Palais--Smale condition provided $\{u_n\}$ is bounded. 
		Suppose, contrary to our claim, that $\|\nabla u_n\|_p \to +\infty$, up to a subsequence, and consider the normalized functions $v_n={u_n}/{\|\nabla u_n\|_p}$. 
		Then, by a standard argument 
		(see, e.g., {\cite[Lemma 3.2]{BobkovTanaka2017}}), it can be shown that 
		$\alpha$ is an eigenvalue of $-\Delta_p$ and 
		$\{v_n\}$ converges strongly  in $\W$ to 
		an eigenfunction $v_0$ corresponding to $\alpha$, up to a subsequence. 
		This gives a contradiction in the case of the assumption~\ref{lem:PS_negative2:1}.
		On the other hand, we have
		\begin{equation*} 
			o(1) 
			=\left\langle \dfrac{\E^\prime(u_n)}{\|\nabla u_n\|_p^{q-1}},\,v_n 
			\right\rangle 
			-\dfrac{p\E(u_n)}{\|\nabla u_n\|_p^q} 
			=\dfrac{q-p}{q} G_\beta(v_n)
			=\dfrac{q-p}{q} G_\beta(v_0)+o(1) 
		\end{equation*} 
		as $n\to+\infty$, which contradicts the second assumption in \ref{lem:PS_negative2:2}.
		Consequently, we proved the boundedness 
		of an arbitrary Palais--Smale sequence of $\E$, which leads to the validity of the Palais--Smale condition. 
	\end{proof}
	
	\begin{lemma}\label{lem:PS_negative} 
		Let either of the assumptions \ref{lem:bdd_below:1}, \ref{lem:bdd_below:2}, \ref{lem:bdd_below:3} of Lemma~\ref{lem:bdd_below} be satisfied.
		Then $\E$ satisfies the Palais--Smale condition 
		at any level $c\in\mathbb{R}$.
	\end{lemma} 
	\begin{proof}
		The result under the assumptions \ref{lem:bdd_below:1} and \ref{lem:bdd_below:2} of Lemma~\ref{lem:bdd_below}, as well as \ref{lem:bdd_below:3} when $\alpha \not\in \sigma(-\Delta_p)$, is a direct corollary of Lemma~\ref{lem:PS_negative2}.
		Let us discuss the case of the assumption~\ref{lem:bdd_below:3} when $\alpha \in \sigma(-\Delta_p)$.
		Let us take any $v\in ES(p;\alpha)\setminus\{0\}$, and so
		$H_\alpha(v)=0$ and $\langle H_\alpha'(v),\xi\rangle = 0$ for any $\xi \in \W$.
		Since $\alpha > \lambda_1(p)$, we have $v_\pm \not\equiv 0$, where $v_\pm = \max\{\pm v, 0\}$. 
		Considering $\xi=v_\pm$, we get $H_\alpha(v_\pm) = 0$.
		Therefore, $v$ and $v_\pm$ are admissible functions for the definition~\eqref{def:beta_*} of $\beta_*(\alpha)$, which gives
		$$
		\beta \leq \beta_*(\alpha) \leq \min\left\{\frac{\|\nabla v\|_q^q}{\|v\|_q^q}, \frac{\|\nabla v_\pm\|_q^q}{\|v_\pm\|_q^q}\right\}.
		$$
		If $\beta < \beta_*(\alpha)$, then $G_\beta(v)>0$. 
		Recalling that $v\in ES(p;\alpha)\setminus\{0\}$ is chosen arbitrarily, we apply Lemma~\ref{lem:PS_negative2}~\ref{lem:PS_negative2:2} to get the desired Palais--Smale condition.
		Let us assume that $\beta = \beta_*(\alpha)$.
		Suppose that $\beta_*(\alpha) = \|\nabla v\|_q^q/\|v\|_q^q$ for some $v\in ES(p;\alpha)\setminus\{0\}$, that is, $v$ is a minimizer of $\beta_*(\alpha)$.
		Then $v_+$ and $v_-$ are also minimizers of $\beta_*(\alpha)$ since
		$$
		\beta_*(\alpha)
		=
		\frac{\|\nabla v\|_q^q}{\|v\|_q^q} 
		= 
		\frac{\|\nabla v_+\|_q^q+\|\nabla v_-\|_q^q}{\|v_+\|_q^q+\|v_-\|_q^q} 
		\geq 
		\min\left\{\frac{\|\nabla v_+\|_q^q}{\|v_+\|_q^q}, \frac{\|\nabla v_-\|_q^q}{\|v_-\|_q^q}\right\} 
		\geq 
		\beta_*(\alpha).
		$$
		Recalling that $\lambda_1(p)<\alpha<\alpha_*$ and $\beta = \beta_*(\alpha)$, \cite[Lemma~5.1]{BobkovTanaka2017} 
		guarantees the existence of $t_\pm>0$ such that both $t_+ v_+$ and $t_- v_-$ are \textit{positive} solutions of \eqref{eq:D}, which is clearly impossible.
		Consequently, we have 
		$$
		\beta = \beta_*(\alpha) < \frac{\|\nabla v\|_q^q}{\|v\|_q^q}
		\quad \text{for any}~ v\in ES(p;\alpha)\setminus\{0\},
		$$
		and hence  $G_\beta(v) > 0$ for any $v\in ES(p;\alpha)\setminus\{0\}$.
		The desired Palais--Smale condition again follows from Lemma~\ref{lem:PS_negative2}~\ref{lem:PS_negative2:2}.
	\end{proof} 
	
	Finally, we discuss a compactness result related to the assumption \ref{lem:bdd_below:2x} of Lemma~\ref{lem:bdd_below}.
	Notice that, unlike Lemma~\ref{lem:bdd_below} \ref{lem:bdd_below:2x}, this result requires the strict inequality $p>2q$, and we do not know if it remains valid when $p=2q$.
	\begin{lemma}\label{lem:PS_negative3} 
		Let $\alpha=\lambda_1(q)$, $\beta=\beta_*$, and $p > 2q$.
		Then $\E$ satisfies the Palais--Smale condition at any level $c < 0$. 
	\end{lemma}
	\begin{proof}
		Let $\{u_n\} \subset \W$ be such that $\E'(u_n) \to 0$ and $\E(u_n) \to c<0$ as $n \to +\infty$.
		As in the proof of Lemma~\ref{lem:PS_negative2}, it is sufficient to show that $\{u_n\}$ is bounded.
		Suppose, by contradiction, that $\|\nabla u_n\|_p \to +\infty$, up to a subsequence.
		This implies that $\|u_n\|_p \to +\infty$. 
		Indeed, if we suppose that $\{\|u_n\|_p\}$ is bounded, then $\{\|u_n\|_q\}$ is also bounded, and hence we get $\E(u_n) \to +\infty$, which contradicts the assumption $\E(u_n) \to c$. 
		Applying \cite[Lemma~3.2]{BobkovTanaka2017}, we deduce that the sequence of normalized functions $w_n = u_n/\|u_n\|_p$ converges  strongly in $\W$ to either $\varphi_p$ or  $-\varphi_p$ (recall that we assume $\|\varphi_p\|_p=1$), up to a subsequence.
		
		Let us now decompose each $w_n$ as $w_n = \gamma_n \varphi_p + v_n$, where $\gamma_n \in \mathbb{R}$ and $v_n \in \W$ satisfy
		$$
		\gamma_n = \frac{\intO w_n \varphi_p \,dx}{\|\varphi_p\|_2^2}
		\quad \text{and} \quad
		\intO \varphi_p v_n \,dx = 0.
		$$ 
		The convergence of $\{w_n\}$ and the assumption $p>2q \,(>2)$ imply that either
		$\gamma_n \to 1$ or $\gamma_n \to -1$, and that $\|\nabla v_n\|_p \to 0$.
		Thanks to the assumption $\alpha = \lambda_1(p)$, we have $\Ha(u_n) \geq 0$. 
		Therefore, since $\E(u_n) \to c<0$, we conclude that $\Gb(u_n)<0$ for all sufficiently large $n$.
		Then Lemma~\ref{lem:HG<0}~\ref{lem:H<02} yields $u_n \in [G_\beta < 0] \cap [H_\alpha > 0]$. 
		According to Proposition~\ref{prop:minpoint}, there exists a unique minimum point $t(u_n)$ of $t \mapsto \E(tu_n)$, and hence
		\begin{equation}\label{eq:jjee}
			\J(w_n) = \J(u_n) \equiv \E(t(u_n)u_n) \leq \E(u_n) = c+o(1) < 0,
		\end{equation}
		where $\J$ is the $0$-homogeneous  functional defined in \eqref{eq:J}.
		Arguing now in much the same way as in the proof of \cite[Proposition~11]{BobkovTanaka2017}  (see, more precisely, the part on pp.~1233-1234) and substituting the improved Poincar\'e inequality from \cite{takac} by \cite[Theorem~1.2]{BK} (which provides the same inequality under weaker assumptions on $\Omega$ than in \cite{takac}), 	
		we deduce that 	
		$$
		\liminf_{n\to+\infty}
		\J(w_n)
		\geq -C \,\limsup_{n\to+\infty} \left( \intO |\nabla \varphi_p|^{p-2} |\nabla v_n|^2 \,dx + \intO |\nabla v_n|^p \,dx \right)^\frac{p-2q}{2(p-q)} = 0,
		$$
		since $\|\nabla v_n\|_p \to 0$ and $p>2q$.	
		This is a contradiction to \eqref{eq:jjee}, and hence $\{u_n\}$ is a bounded Palais--Smale sequence for $\E$.
	\end{proof}

	\section{Proofs of the main results}\label{sec:proofs}
	
	\subsection{Proof of Theorem~\ref{thm1}} 
	
	In the case of the assumptions~\ref{thm1:1},  \ref{thm1:2}, \ref{thm1:2x} of Theorem~\ref{thm1}, the existence of at least $k$ distinct pairs of critical points of $\E$ in $[\E<0]$ can be established by standard methods (see, e.g.,  \cite[Theorem~8]{Clark}), since $\E$ is bounded from below in $\W$ and satisfied the Palais--Smale condition thanks to the results of Section~\ref{sec:propertiesE}.
	However, in the case of the assumption~\ref{thm1:3}, i.e., when $\alpha > \lambda_1(p)$ and $\lambda_k(q)<\beta \leq \beta_*(\alpha)$, $\E$ is not bounded from below anymore, 
	and we need to argue differently. 
	Our approach will be based on the analysis of the restriction of $\E$ to the sublevel set $[G_\beta<0]$, and it will cover all the assumptions~\ref{thm1:1}, \ref{thm1:2}, \ref{thm1:2x}, \ref{thm1:3} in a unified way.
	
	\begin{proof*}{Theorem~\ref{thm1}}
		For $j=1,\ldots,k$, we define 
		\begin{equation}\label{def:Thm1-b} 
			a_j 
			=
			\inf
			\left\{
			\max_{u\in A} \E(u):~ A\in {\Sigma}_j(p)~ \text{and}~ A \subset [G_\beta<0]
			\right\},
		\end{equation} 
		where ${\Sigma}_j(p)$ is given by \eqref{eq:Sigma-r}.
		Thanks to Lemmas~\ref{lem:bdd_below} and \ref{lem:surface}, 
		we have 
		$$
		-\infty<a_1\le a_2\le \cdots \le a_k<0.
		$$ 
		Let us show that each $a_j$ is a critical value of $\E$. 
		Suppose, by contradiction, that some $a_j$ is a regular value of $\E$. 
		Since $\E$ satisfies the Palais--Smale condition at any negative level by Lemma~\ref{lem:PS_negative} (under either of the assumptions~\ref{thm1:1}, \ref{thm1:2}, \ref{thm1:3}) or Lemma~\ref{lem:PS_negative3} (under the assumption~\ref{thm1:2x}), 
		the deformation lemma (see, e.g., \cite[Chapter~II, Theorem 3.4 and Remark~3.5]{Struwe}) 
		guarantees the existence of $\varepsilon>0$ and $\eta\in C([0,1]\times \W,\W)$ 
		with the following properties: 
		\begin{itemize} 
			\item[{\rm (i)}] $\eta(0,u)=u$ for any $u\in \W$; 
			\item[{\rm (ii)}] $\E(\eta(t,u))$ is nonincreasing with respect to $t$ for any $u\in\W$; 
			\item[{\rm (iii)}] $\eta(1,u)\in [\E\le a_j-\varepsilon]$ provided 
			$u\in [\E \le a_j+\varepsilon]$; 
			\item[{\rm (iv)}] $\eta(t,u)$ is odd with respect to $u$ for any $t\in[0,1]$. 
		\end{itemize} 
		For $\varepsilon>0$ as above, by the definition of $a_j (<0)$, 
		we can choose $A$ such that $A \in {\Sigma}_j(p)$, $A \subset [G_\beta<0]$, and 
		$\max_{v\in A} \E(v)< \min\{a_j+\varepsilon,\,0\}$. 
		Now we define
		$$
		\tilde{A} =\eta(1,A) =\{\eta(1,u):~ u\in A\}. 
		$$
		Clearly, $\tilde{A}$ is symmetric and compact in $\W$. 
		Moreover, we have $\gamma(\tilde{A})\ge \gamma(A)\ge j$ by, e.g., Lemma~\ref{lem:genus}. 
		Therefore, if we can prove that
		$\tilde{A}\subset [G_\beta<0]$, then 
		$\tilde{A}$ is an admissible set for the definition \eqref{def:Thm1-b} of $a_j$, and we arrive at a contradiction 
		as follows: 
		$$
		a_j\le \max_{v\in \tilde{A}}\E(v)
		=\max_{u\in A} \E(\eta(1,u)) \le a_j-\varepsilon.  
		$$
		Let us show that $\tilde{A}\subset [G_\beta<0]$. 
		Suppose, by contradiction, that there exists $u\in A$ such that 
		$\eta(1,u)\in [G_\beta\ge 0]$. 
		Since $\eta$ is continuous and $u\in A\subset [G_\beta<0]$, we can find 
		$t_0\in (0,1]$ such that $\eta(t_0,u)\in [G_\beta=0]$.
		Thanks to Lemma~\ref{lem:HG<0} \ref{lem:H<03}, we have $\E(\eta(t_0,u)) \geq 0$.
		However, this
		contradicts to 
		$$
		\E(\eta(t_0,u))\le \E(u)\le \max_{v\in A}\E(v) 
		<\min \{a_j+\varepsilon,\, 0\}.
		$$
		Consequently, we have shown that each $a_j$  is a critical value of $\E$. 
		If all these values are distinct, we obtain at least $k$ distinct critical points (and hence distinct \textit{pairs} of critical points) of $\E$, which completes the proof.
		Assume now that some critical values coincide, that is, 
		$$
		a:=a_j=a_{j+1}=\cdots=a_{j+l} ~(<0) 
		$$
		for some $j \in \{1,\dots,k-1\}$ and $l \in \{1,\dots,k-j\}$. 
		Since $\E$ satisfies the Palais--Smale condition at $a$,  
		and $K_a \cap [\E<0]\subset [G_\beta<0]$ (see \eqref{eq:EHG}), where $K_a$ is the critical set of $\E$ at the level $a$, it can be proved by standard arguments as, e.g, in  \cite[Chapter~II, Lemma~5.6]{Struwe} that $\gamma(K_a)\ge l+1$. Consequently, by \cite[Chapter~II, Observation~5.5]{Struwe}, $K_a$ contains infinitely many distinct pairs of critical points of $\E$. 
	\end{proof*}
	
	\begin{remark} 
		Under the assumption $[G_\beta<0] \neq \emptyset$, we have $a_1>-\infty$ if and only if $\inf_{\N} \E > -\infty$, as it follows from the proof of Lemma~\ref{lem:bdd_below}.
		Therefore, Theorem~\ref{thm1} can also be proved by considering the restriction of $\E$ to the Nehari manifold $\N$. 
		We proceed in this way in order to prove Theorem~\ref{thm3}. 
	\end{remark}

	\subsection{Proof of Theorem~\ref{thm2}}
	
	For convenience, we start by providing the following known auxiliary result.
	\begin{lemma}[\protect{\cite[Lemma~3.2]{BobkovTanakaNodal1}}]\label{lem:bdd_below_Y} 
		Let $\alpha,\beta\in\mathbb{R}$ and $\lambda>\max\{0,\alpha\}$. 
		Then $\E$ is bounded from below on 
		the set 
		\begin{equation}\label{def:Y} 
			Y(\lambda) = \{u\in\W:~ \|\nabla u\|_p^p \ge \lambda \|u\|_p^p \}. 
		\end{equation}
	\end{lemma} 
	
	\begin{proof*}{Theorem~\ref{thm2}} 
		For $j=l,\ldots,l+k-1$, we define 
		\begin{align} 
			b_j :=
			\inf
			\left\{
			\max_{u\in A} \E(u):~ A\in \Sigma_j(p)
			\right\},
		\end{align} 
		where $\Sigma_j(p)$ is given by \eqref{eq:Sigma-r}.
		Let us show that 
		$$
		-\infty <b_{l}\le \cdots \le b_{k+l-1}<0. 
		$$
		First, we prove the lower bound $-\infty <b_{l}$.	
		By the definition \eqref{eq:lambdak} of $\lambda_{l}(p)$, for any $A\in \Sigma_{l}(p)$ 
		we can find $u_0 \in A$ such that 
		$\|\nabla u_0\|_p^p\ge \lambda_{l}(p)\|u_0\|_p^p$. 
		This means that $u_0\in Y(\lambda_{l}(p))$, 
		where the set $Y(\lambda_{l}(p))$ is 
		defined by \eqref{def:Y}. 
		Hence, recalling the assumption $\alpha<\lambda_{l}(p)$ and using Lemma~\ref{lem:bdd_below_Y}, we obtain the inequalities
		$$
		\max_{u\in A}\E(u) \ge \E(u_0) \ge \inf\left\{\E(v):~ v \in Y(\lambda_{l}(p))\right\} >-\infty,
		$$
		which imply the desired lower bound $-\infty <b_{l}$.
		Second, we justify the upper bound $b_{k+l-1}<0$.
		Since $\beta>{\lambda}_{k+l-1}(q)$, 
		Lemma~\ref{lem:surface} gives the existence of 
		$A\in {\Sigma}_{k+l-1}(p)$ such that $A \subset [G_\beta < 0]$ and $\max_{u\in A}\E(u)<0$, and hence
		$b_{k+l-1}\le \max_{u\in A}\E(u)<0$.
		
		Finally, due to Lemma~\ref{lem:PS_negative2}, $\E$ satisfies the Palais--Smale condition at any level 
		$c\in\mathbb{R}$ under the imposed assumptions on $\alpha$ and $\beta$, and hence standard arguments based on the deformation lemma (cf.\ the proof of Theorem~\ref{thm1}) guarantee that every $b_j$ is a negative critical value of $\E$. 
		In the same way in the proof of Theorem~\ref{thm1}, we also conclude that there exist at least $k$ distinct pairs of critical points of $\E$. 
	\end{proof*}

	%%%%%%%%%%%%%%%%%%%%%%%%%%%%%%%%%%%%%%%%%%%%%%%
	%%%%%%%%%%%%%%%%%%%%%%%%%%%%%%%%%%%%%%%%%%%%%%%%
	%%%%%%%%%%%%%%%%%%%%%%%%%%%%%%%%%%%%%%%%%%%%%%%%%
	\subsection{Proof of Theorem~\ref{thm3}} 
	
	The proof will be based on the application of the following general theorem to 
	the functional $\E$ on the manifold $\N\cap B_{\alpha,\beta}^+$ introduced in Section~\ref{sec:nehari}. 
	
	\begin{theorem}[{\cite[Corollary~4.1]{Szulkin}}]\label{thm:Szulkin} 
		Let $M$ be a closed symmetric $C^1$-submanifold of a real Banach space $X$ and $0\not\in M$. 
		Assume that a functional $I \in C^1(M,\mathbb{R})$ is even and bounded from below. Define 
		\begin{align*} 
			c_j
			&=
			\inf
			\left\{
			\max_{u\in A} I(u):~ A\in \Gamma_j
			\right\},\\ 
			\Gamma_j &=\left\{A\subset M:~ A~\mathrm{is~  symmetric, \, compact \ in}~  X, ~\mathrm{and}~ 
			\gamma(A)\ge j\,\right\}. 
		\end{align*}
		If $\Gamma_l\not=\emptyset$ for some $l \in \mathbb{N}$ and if 
		$I$ satisfies the Palais--Smale condition at the level $c_j$ for any $j \in \{1,\dots,l\}$, 
		then every $c_j$ is a critical value of $I$, 
		and hence $I$ has at least $l$ distinct pairs of critical points. 
	\end{theorem}

	%%%%%%%%%%%%%%%%%%%%%%%%%%%%%%%%%%%%%%%%%
	%%%%%%%%%%%%%%%%%%%%%%%%%%%%%%%%%%%%%%%%%%
	\begin{proof*}{Theorem~\ref{thm3}}
		For $j=1,\ldots,l$, we define 
		\begin{align} 
			c_j 
			&=
			\inf
			\left\{
			\max_{u\in A} \EN(u):~ A\in \Gamma_j
			\right\}, 
			\label{def:Thm2-c} \\ 
			\Gamma_j &=\left\{A\subset \N\cap B_{\alpha,\beta}^+:~A ~\text{is symmetric, compact in}~ \W, ~\text{and}~ \gamma(A)\ge j\, \right\} 
			\label{def:Thm2-set}. 
		\end{align} 
		Since we assume that $\alpha>\lambda_l(p)$ and $\beta < \beta_*(\alpha)$, 
		Proposition~\ref{prop:LeastEnergy} yields 
		$$
		0 < c_1 \leq \cdots \leq c_l.
		$$ 
		Let us show that $\Gamma_l \neq \emptyset$. 
		By the definition \eqref{eq:lambdak} of $\lambda_l(p)$ and the assumption $\alpha>\lambda_l(p)$, there exists
		$A\in \Sigma_l(p)$ satisfying 
		\begin{equation}\label{eq:uh0g}
			\max_{u\in A} \frac{\|\nabla u\|_p^p}{\|u\|_p^p} <\alpha,
			~~\text{and hence}~
			\max_{u\in A} H_\alpha(u) < 0 < \min_{u\in A} G_\beta(u),
		\end{equation}
		where the last inequality follows from
		Lemma~\ref{lem:HG<0}~\ref{lem:H<0}.
		Consequently, for any $u \in A$ we have $t(u)u \in \N\cap B_{\alpha,\beta}^+$, 
		where $t(u)>0$ is defined by \eqref{tu}, see Proposition~\ref{prop:minpoint}.
		Let us observe that the mapping $u \mapsto t(u) u$ is odd and continuous for $u \in A$, as it follows from \eqref{tu} and the second part of \eqref{eq:uh0g}.
		Therefore, considering
		$$
		\widetilde{A}=\left\{t(u)u:~ u\in A\right\} \subset \N\cap B_{\alpha,\beta}^+,
		$$
		we deduce that $\widetilde{A}$ is symmetric and compact in $\W$, and 
		$\gamma(\widetilde{A}) \geq \gamma(A)\ge l$ by, e.g.,
		Lemma~\ref{lem:genus}.
		As a result, $\widetilde{A}\in \Gamma_l$, i.e., $\Gamma_l \neq \emptyset$. 
		Thanks to Proposition~\ref{prop:MFD} and Lemma~\ref{lem:PS_positive}, 
		Theorem~\ref{thm:Szulkin} leads to the existence of 
		$l$ distinct pairs of critical points of $\EN$ in $[\EN>0]$. 
		Finally, applying Proposition~\ref{prop:critical}, we conclude that these critical points of $\EN$ are also critical points of the original functional $\E$, which completes the proof.
	\end{proof*}
	%%%%%%%%%%%%%%%%%%%%%%%%%%%%%%%%%%%%%%%%%%%%%%%
	%%%%%%%%%%%%%%%%%%%%%%%%%%%%%%%%%%%%%%%%%%%%%%%

	\appendix
	\section{Characterizations of \texorpdfstring{$\lambda_k(q)$}{lambda-k(q)}}\label{sec:appendixA}
	
	In this section, we discuss alternative characterizations of variational eigenvalues $\lambda_k(q)$ using narrower and larger constraint sets than in the standard minimax definition \eqref{eq:lambdak}.
	The results of this section require less restrictive assumptions on a bounded domain than that additionally imposed in Section~\ref{sec:intro}. 
	Because of this, we will provide explicit assumptions on $\Omega$ in each statement.
	
	Let $q, r>1$.
	Recall the definition \eqref{eq:lambdak} of $\lambda_k(q)$:
	\begin{equation*}
		\lambda_k(q) 
		=
		\inf
		\left\{
		\max_{u\in A} \frac{\|\nabla u\|_q^q}{\|u\|_q^q}:~ A\in \Sigma_k(q)
		\right\},
	\end{equation*}
	and consider a related object $\lambda_k(q;r)$ defined as
	\begin{align}\label{eq:lambdak-tilde}
		\lambda_k(q;r)
		=
		\inf
		\left\{
		\sup_{u\in A} \frac{\|\nabla u\|_q^q}{\|u\|_q^q}:~ A\in {\Sigma}_k(r)
		~\text{and}~
		A \subset W^{1,q}
		\right\},
	\end{align}
	where $\Sigma_k(r)$ is given by \eqref{eq:Sigma-r}, i.e.,
	\begin{equation*}
		\Sigma_k(r) =\left\{\, A\subset W_0^{1,r} \setminus \{0\}:~
		A\ {\rm is\ symmetric,\ compact\ in}\ W^{1,r}_0,\ 
		{\rm and}\ \gamma(A;\Wr)\ge k\,\right\}.
	\end{equation*}
	Throughout this section, we use the expanded notation $\gamma(A;X)$ for the Krasnoselskii genus to emphasize its dependence on a topological vector space $X$ with respect to which the continuity of an odd mapping $h$ in the definition \eqref{eq:genus} of the genus is understood.
	
	Clearly, $\lambda_k(q;q)=\lambda_k(q)$ for any $q>1$, while $\lambda_k(q;r)$ is defined using different constraints than $\lambda_k(q)$ whenever $q \neq r$.
	The constraint $A \subset W^{1,q}$ in \eqref{eq:lambdak-tilde} is trivial in the case $q<r$ since $\Wr \subset \Wq$, but it is required when $q>r$, in order for the Rayleigh quotient $\|\nabla u\|_q^q/\|u\|_q^q$ to be defined over $A\in {\Sigma}_k(r)$.
	
	We are interested in the relation between $\lambda_k(q;r)$ and $\lambda_k(q)$.
	Let us start with the following observation on the Krasnoselskii genus.
	\begin{lemma}\label{lem:genus}
		Let $X, Y$ be topological vector spaces.
		Let $A \subset X \setminus \{0\}$ be symmetric and closed.
		Let $\pi: A \to Y$ be odd and continuous.
		Let $\pi(A) \subset Y \setminus \{0\}$ be (symmetric and) closed.
		Then 	
		$\gamma(A;X) \leq \gamma(\pi(A);Y)$.
		If, in addition, the inverse mapping $\pi^{-1} : \pi(A) \to A$ is continuous (and hence $\pi(A)$ and $A$ are homeomorphic), then $\gamma(A;X) = \gamma(\pi(A);Y)$.
	\end{lemma}
	\begin{proof}
		If  $\gamma(\pi(A);Y) = +\infty$, then the inequality $\gamma(A;X) \leq \gamma(\pi(A);Y)$ is trivial.
		Assume that $k=\gamma(\pi(A);Y) < +\infty$.
		By the definition \eqref{eq:genus} of $\gamma(\pi(A);Y)$, there exists  an odd mapping $h \in C(\pi(A);\mathbb{R}^k \setminus\{0\})$.
		Since $\pi$ is odd and continuous, we see that $h \circ \pi$ is odd and $h \circ \pi \in C(A;\mathbb{R}^k \setminus\{0\})$. 
		Therefore, $h \circ \pi$ is admissible for the definition \eqref{eq:genus} of $\gamma(A;X)$, which implies that $\gamma(A;X) \leq \gamma(\pi(A);Y)$.
		If, in addition, $\pi^{-1}: \pi(A) \to A$ is continuous, then the same arguments give $\gamma(\pi(A);Y) \leq \gamma(A;X)$, which yields the equality between genuses.
	\end{proof}
	\begin{remark}
		Evidently, if $A$ in Lemma~\ref{lem:genus} is compact in $X$, then $\pi(A)$ is compact in $Y$. 
		We also refer to \cite[Proposition~7.5:~$2^\circ$]{rabinowitz} or \cite[Chapter~II, Proposition~5.4~($4^\circ$)]{Struwe} for the statement of Lemma~\ref{lem:genus} in the case when $X$ and $Y$ coincide.
	\end{remark}
	
	In view of the continuity of the canonical embedding $i: \Wr \to \Wq$ for $1<q<r$ defined as $i(u)=u$, Lemma~\ref{lem:genus} implies that 
	$$
	k \leq \gamma(A;\Wr) \leq \gamma(A;\Wq) \quad \text{for any}~ A \in \Sigma_k(r),
	$$
	which yields $A \in \Sigma_k(q)$.
	A similar relation holds in the case $q>r>1$. 
	This leads to the following remark.
	\begin{remark}\label{rem:append}
		The following assertions hold:
		\begin{enumerate}[label={\rm(\roman*)}]
			\item\label{rem:append:1} If $1<q<r$, then $\Sigma_k(q) \supset \Sigma_k(r)$ and hence $\lambda_k(q) \leq \lambda_k(q;r)$.
			\item\label{rem:append:2} If $q>r>1$, then
			$\Sigma_k(q) \subset \Sigma_k(r)$ and hence $\lambda_k(q) \geq \lambda_k(q;r)$.
		\end{enumerate}
	\end{remark}
	
	\subsection{The case \texorpdfstring{$q<r$}{q<r}}\label{sec:appA1}
	First, we show that $\lambda_k(q) = \lambda_k(q;r)$ whenever $q<r$, which therefore provides an alternative characterization of $\lambda_k(q)$.
	In the case $k=1$, this claim is simple thanks to the density of $C_0^\infty(\Omega)$ in both $\Wr$ and $\Wq$, and our aim is to develop this approach for higher indices $k$.
	
	\begin{lemma}\label{lem:la=la}
		Let $\Omega \subset \mathbb{R}^N$ be a bounded domain, $N \geq 1$.
		Let $1<q<r$. Then 
		${\lambda}_k(q) = \lambda_k(q;r)$ for all $k \in \mathbb{N}$.
	\end{lemma}
	\begin{proof}
		We know from Remark~\ref{rem:append}~\ref{rem:append:1} that 
		${\lambda}_k(q)\le \lambda_k(q;r)$ for all $k$. 
		Suppose, by contradiction, that ${\lambda}_k(q) < \lambda_k(q;r)$ for some $k$.
		Consequently, there exists $A \in \Sigma_k(q)$ such that
		\begin{equation}\label{eq:upper-bound0}
			\lambda_k(q)
			\leq
			\max_{u\in A} \frac{\|\nabla u\|_q^q}{\|u\|_q^q}
			<
			\lambda_k(q;r).
		\end{equation}
		To prove the lemma, we show the existence of an element of $\Sigma_k(r)$ sufficiently close to $A$ in the norm of $W^{1,q}$, which will give a contradiction to \eqref{eq:upper-bound0}.
		
		Let us take any $\varepsilon \in (0,\text{dist}_q(A,0)/2)$, where
		$$
		\text{dist}_q(A,0)
		=
		\min_{u \in A}\|\nabla u\|_q  > 0
		$$
		since $A$ is compact in $W_0^{1,q}$ and does not contain $0$.
		Let $B_\varepsilon^q(u)$ be an open ball in $W_0^{1,q}$ of radius $\varepsilon>0$ centered at $u$.
		Thanks to the compactness of $A$ in $W_0^{1,q}$,
		we can extract a finite symmetric subcover $\{B_\varepsilon^q(u_n) \cup B_\varepsilon^q(-u_n)\}_{n=1}^K$ from a cover $\cup_{u \in A}B_\varepsilon^q(u)$, where $K= K(\varepsilon) \geq 1$.	
		
		For every $n \in \{1,\dots,K\}$ we choose some $v_n \in C_0^\infty(\Omega) \cap B_\varepsilon^q(u_n)$.
		In particular, we have 
		$\|\nabla(u - v_n)\|_q < 2\varepsilon$ for any $u \in B_\varepsilon^q(u_n)$.
		Consider a closed finite dimensional linear subspace of $W_0^{1,q}$ spanned by $\{v_n\}_{n=1}^K$:
		$$
		V = V(\varepsilon) = \text{span} \{v_1,\dots,v_K\},
		$$
		and define the metric projection $P_\varepsilon: W_0^{1,q} \mapsto V$ in the standard way:
		$$
		\|\nabla (u-P_\varepsilon(u))\|_q = \inf_{v \in V}\|\nabla (u-v)\|_q, \quad u \in W_0^{1,q}.
		$$
		Our aim is to prove that $P_\varepsilon(A) \in \Sigma_k(r)$, after a possible decrease in the value of $\varepsilon \in (0,\text{dist}_q(A,0)/2)$.
		The operator $P_\varepsilon$ is well defined in the sense that $P_\varepsilon(u)$ exists and unique for any $u \in W_0^{1,q}$, see, e.g., \cite[Corollary 3.4 $(1^\circ)$, $(3^\circ)$, p.~111]{singer}. 
		Moreover, $P_\varepsilon$ is a continuous mapping, see, e.g., \cite[Theorem~5.4, p.~251]{singer}.
		It is also clear that $P_\varepsilon$ is odd, i.e., $P_\varepsilon(-u)=-P_\varepsilon(u)$. 
		As a consequence of the last two facts, we have $\gamma(P_\varepsilon(A);\Wq) \geq k$, see, e.g., Lemma~\ref{lem:genus}. 
		The continuity of $P_\varepsilon$ and the compactness of $A$ in $\Wq$ imply that $P_\varepsilon(A)$ is compact in $W_0^{1,q}$.
		Moreover, $0 \not\in P_\varepsilon(A)$ thanks to the upper bound on $\varepsilon$.
		It is clear from the regularity of the basis elements of $V$ that $V$ is a closed finite dimensional linear subspace of $\Wr$, and hence $P_\varepsilon(A) \subset \Wr$. 
		Since in the finite dimensional space $V$ all norms are equivalent, we conclude that $P_\varepsilon(A)$ is compact in $\Wr$. 
		Applying Lemma~\ref{lem:genus} with the inclusion mapping $i:V \subset \Wq \to \Wr$ defined as $i(u)=u$, we get $\gamma(P_\varepsilon(A);\Wr) = \gamma(P_\varepsilon(A);\Wq) \geq k$, which finishes the proof that $P_\varepsilon(A) \in \Sigma_k(r)$.
		
		Finally, let us explicitly obtain an upper bound for the Rayleigh quotient for $\lambda_k(q;r)$ to get a contradiction to \eqref{eq:upper-bound0}. 	
		For any $u \in A$ there exists  $n \in \{1,\dots,K\}$ such that $u \in B_\varepsilon^q(u_n)$.
		By the triangle inequality and the definition of $P_\varepsilon(u)$, we have
		\begin{equation}\label{eq:upper-bound11}
			\left|
			\|\nabla u\|_q
			-
			\|\nabla P_\varepsilon(u)\|_q
			\right|
			\leq
			\|\nabla (u-P_\varepsilon(u))\|_q
			= \inf_{v \in V}\|\nabla (u-v)\|_q
			\leq 
			\|\nabla (u-v_n)\|_q 
			< 2\varepsilon,
		\end{equation}
		and hence
		\begin{equation}\label{eq:upper-bound1}
			\|\nabla P_\varepsilon(u)\|_q 
			\leq 
			\|\nabla u\|_q
			+
			2\varepsilon
			\quad \text{for any}~ u \in A.
		\end{equation}
		The inequality \eqref{eq:upper-bound11} 
		and the Poincar\'e inequality give
		\begin{equation}\label{eq:upper-bound31}
			\left| \|u\|_q - \|P_\varepsilon(u)\|_q\right|
			\leq 
			\|u-P_\varepsilon(u)\|_q \leq \lambda_1^{-1/q}(q) \, \|\nabla (u-P_\varepsilon(u))\|_q
			\leq 
			2 \lambda_1^{-1/q}(q) \varepsilon
		\end{equation}
		for any $u \in A$, 
		which yields
		\begin{equation}\label{eq:upper-bound3}
			1 
			\leq
			\frac{\|P_\varepsilon(u)\|_q}{\|u\|_q} + \frac{2 \lambda_1^{-1/q}(q)\varepsilon}{\|u\|_q}.
		\end{equation}
		Combining \eqref{eq:upper-bound1} and \eqref{eq:upper-bound3}, we get
		$$
		\|\nabla P_\varepsilon(u)\|_q
		\leq
		\|\nabla u\|_q
		+2\varepsilon
		\leq
		\|\nabla u\|_q
		\frac{\|P_\varepsilon(u)\|_q}{\|u\|_q}
		+
		\|\nabla u\|_q
		\frac{2 \lambda_1^{-1/q}(q)\varepsilon}{\|u\|_q}
		+2\varepsilon.
		$$
		Recalling that $0 \not\in P_\varepsilon(A)$ and dividing by $\|P_\varepsilon(u)\|_q$, we arrive at
		\begin{equation}\label{eq:upper-bound4}
			\frac{\|\nabla P_\varepsilon(u)\|_q}{\|P_\varepsilon(u)\|_q}
			\leq
			\frac{\|\nabla u\|_q}{\|u\|_q}
			+
			\frac{\|\nabla u\|_q}{\|u\|_q}
			\cdot
			\frac{2 \lambda_1^{-1/q}(q)\varepsilon}{\|P_\varepsilon(u)\|_q}
			+
			\frac{2\varepsilon}{\|P_\varepsilon(u)\|_q}.
		\end{equation}
		Moreover, we also deduce from \eqref{eq:upper-bound31} that
		\begin{equation}\label{eq:upper-bound5}
			\|P_\varepsilon(u)\|_q 
			\geq 
			\|u\|_q - 2 \lambda_1^{-1/q}(q) \varepsilon 
			\geq
			\min_{v \in A} \|v\|_q - 2 \lambda_1^{-1/q}(q) \varepsilon.
		\end{equation}
		Combining \eqref{eq:upper-bound4} and \eqref{eq:upper-bound5} and taking $\varepsilon>0$ smaller if necessary, we conclude that
		$$
		\lambda_k^{1/q}(q;r) 
		\leq 
		\max_{u \in A}\frac{\|\nabla P_\varepsilon(u)\|_q}{\|P_\varepsilon(u)\|_q}
		\leq
		\max_{u \in A}
		\frac{\|\nabla u\|_q}{\|u\|_q} + O(\varepsilon)
		< 
		\lambda_k^{1/q}(q;r),
		$$
		where the last inequality is given by our contradictory assumption \eqref{eq:upper-bound0}. 
		A contradiction.
	\end{proof}
	
	\begin{remark}\label{rem:gap}
		Lemma~\ref{lem:la=la} is needed for the proofs of Theorems~\ref{thm1} and~\ref{thm2} (see, more explicitly, the proof of Lemma~\ref{lem:surface}), cf.\ \cite{ZongoRuf2}.
	\end{remark}

	\subsection{The case \texorpdfstring{$q>r$}{q>r}}\label{sec:appA2}
	
	Let us now discuss the relation between $\lambda_k(q;r)$ and ${\lambda}_k(q)$ in the case $q<r$, which is converse to that considered above.
	Although $\lambda_k(q;r) \leq {\lambda}_k(q)$ for all $k$ (see Remark~\ref{rem:append}~\ref{rem:append:2}), the precise relation between $\lambda_k(q;r)$ and ${\lambda}_k(q)$ heavily depends on the regularity of $\Omega$ and it is closely connected to the problem of continuity of the mapping $q \mapsto \lambda_k(q)$ from the left, see \cite[Section~7]{Lind1},  \cite{DM,HK,parini}, and references therein.
	We provide two opposite results in this direction.
	
	\begin{lemma}
		Let $N \geq 2$ and $1<r<q$.
		Let $\Omega \subset \mathbb{R}^N$ be a bounded domain such that 
		\begin{equation}\label{eq:WWW0}
			W_0^{1,q_-}
			:=
			W^{1,q} \cap \bigcap_{1<s<q} W_0^{1,s} \neq W_0^{1,q}.
		\end{equation}
		Then $\lambda_k(q;r) < {\lambda}_k(q)$ for any $k \in \mathbb{N}$.
	\end{lemma}
	\begin{proof}
		It is not hard to observe
		that $W_0^{1,q_-}$ is a closed vector subspace of $W^{1,q}$ satisfying $W_0^{1,q} \subset W_0^{1,q_-} \subset \Wr$ for any $r<q$, see, e.g., \cite[Proposition~2.1]{DM}.
		Define
		$$
		\underline{\lambda}_k(q)
		=
		\inf
		\left\{
		\max_{u\in A} \frac{\|\nabla u\|_q^q}{\|u\|_q^q}:~ A\in {\Sigma}_k(q_-)
		\right\},
		$$
		where
		$$
		\Sigma_k(q_-) 
		=\left\{\, A\subset W_0^{1,q_-} \setminus \{0\}:~ 
		A\ {\rm is\ symmetric,\ compact\ in}\ W^{1,q},\ 
		{\rm and}\ \gamma(A;W^{1,q})\ge k\,\right\}.
		$$
		Using the continuous embedding $i_1: \Wq \to W^{1,q}$ defined as $i_1(u)=u$, we apply Lemma~\ref{lem:genus} to deduce that $\Sigma_k(q) \subset \Sigma_k(q_-)$. 
		In the same way, Lemma~\ref{lem:genus} with the continuous embedding $i_2: W^{1,q} \to W^{1,r}$ defined as $i_2(u)=u$ yields $\gamma(A;W^{1,r}) \geq \gamma(A;W^{1,q}) \geq k$ for any $A \in \Sigma_k(q_-)$.
		Moreover, any such $A$ is symmetric and compact in $\Wr$.
		Since the topologies induced by the norms of $W^{1,r}$ and $\Wr$ are equivalent in $\Wr$, we conclude from Lemma~\ref{lem:genus} that $\gamma(A;W^{1,r})=\gamma(A;\Wr)$ for any $A \in \Sigma_k(q_-)$.
		Therefore, $\Sigma_k(q_-) \subset \Sigma_k(r)$.
		Consequently, 
		\begin{equation}\label{eq:lalala1}
			\lambda_k(q;r) 
			\leq 
			\underline{\lambda}_k(q) 
			\leq 
			\lambda_k(q).
		\end{equation}
		It follows from \cite[Theorem~6.1 and Corollary~6.2]{DM} that $\underline{\lambda}_k(q) 
		= \lambda_k(q)$ for $k \geq 2$ if and only if $\underline{\lambda}_1(q) 
		= \lambda_1(q)$.
		However, our assumption \eqref{eq:WWW0}  
		is equivalent to 
		$\underline{\lambda}_1(q) 
		< \lambda_1(q)$, see \cite[Theorem~4.1~(c),~(d)]{DM}.
		Thus, we conclude from \eqref{eq:lalala1} that $\lambda_k(q;r) \leq 
		\underline{\lambda}_k(q)  < \lambda_k(q)$ for any $k \in \mathbb{N}$.
	\end{proof}
	
	\begin{remark}
		A domain $\Omega$ satisfying the assumption \eqref{eq:WWW0}, or, equivalently, 	
		$\lim_{s \nearrow q} \lambda_1(s) < \lambda_1(q)$, was first constructed in \cite[Section~7]{Lind1}. 
		We refer to \cite{DM,HK} for further developments and for characterizations of the family of such domains. 
	\end{remark}

	\begin{lemma}\label{lem:la=la2}
		Let $N \geq 1$ and $1<r<q$.
		Let $\Omega \subset \mathbb{R}^N$ be a bounded domain such that
		\begin{equation}\label{eq:WWW}
			W^{1,q} \cap \Wr = \Wq.
		\end{equation}
		Then
		${\lambda}_k(q) = \lambda_k(q;r)$ for all $k \in \mathbb{N}$.
	\end{lemma}
	\begin{proof}
		Let us take any $\varepsilon>0$ and any $A_\varepsilon \in \Sigma_k(r)$ such that $A_\varepsilon \subset W^{1,q}$ and
		\begin{equation}\label{eq:upper-bound0x}
			\sup_{u\in A_\varepsilon} \frac{\|\nabla u\|_q^q}{\|u\|_q^q}
			\leq
			\lambda_k(q;r) + \varepsilon.
		\end{equation}
		It is known from \cite{CdP} (see, more precisely, \cite[Eq.~(2.3) and Corollary~3.6]{CdP}) that $\lambda_k(q)$ can be equivalently defined as
		\begin{equation}\label{eq:lambdak2}
			\lambda_k(q) 
			=
			\min
			\left\{
			\sup_{u\in A} \frac{\|\nabla u\|_q^q}{\|u\|_q^q}:~ A\in \mathcal{G}_k(q)
			\right\},
		\end{equation}
		where 
		\begin{align*}
			&\mathcal{G}_k(q) 
			\\
			&=\left\{\, A\subset W_0^{1,q} \setminus \{0\}:~
			A\ {\rm is\ symmetric,\ closed \ and \ bounded\ in}\ W^{1,q}_0,\ 
			{\rm and}\ \gamma(A; L^q)\ge k\,\right\}.
		\end{align*}
		Our aim is to show that $A_\varepsilon \in \mathcal{G}_k(q)$, i.e., $A_\varepsilon$ is admissible for the definition \eqref{eq:lambdak2} of $\lambda_k(q)$. 
		\textit{If it is true}, then $\lambda_k(q) \leq \lambda_k(q;r)+\varepsilon$ in view of \eqref{eq:upper-bound0x}. 
		Since $\varepsilon>0$ is chosen arbitrarily, we get $\lambda_k(q) \leq \lambda_k(q;r)$.
		On the other hand, we have
		$\lambda_k(q;r) \leq {\lambda}_k(q)$ by Remark~\ref{rem:append}~\ref{rem:append:2}, which gives the desired conclusion of the lemma.
		(If we suppose that for any $\varepsilon>0$ there exists $A_\varepsilon$ which is \textit{compact} in $\Wq$, then the above argument applies to the original definition~\eqref{eq:lambdak} of $\lambda_k(q)$ instead of \eqref{eq:lambdak2}, which would give the result in a simpler way.)
		
		Therefore, let us justify that $A_\varepsilon \in \mathcal{G}_k(q)$.
		By the assumption \eqref{eq:WWW}, we have  $A_\varepsilon \subset \Wq$.
		Since $r<q$ and $A_\varepsilon \subset \Wr \setminus \{0\}$ is compact in $\Wr$, we get $A_\varepsilon \subset \Wq \setminus \{0\}$.
		The set $A_\varepsilon$ is bounded in $\Wq$.
		Indeed, we have from  \eqref{eq:upper-bound0x} that
		\begin{equation}\label{eq:nablavbound}
			\|\nabla v\|_q^q 
			\leq 
			(\lambda_k(q;r) + \varepsilon) \|v\|_q^q
			\quad \text{for any}~ v \in A_\varepsilon,
		\end{equation} 
		and hence \cite[Lemma~9]{T-2014} and the Poincar\'e inequality give the existence of a constant $C>0$ such that
		\begin{equation}\label{eq:nablavbound2}
			\|\nabla v\|_q 
			\leq 
			C \|v\|_r
			\leq
			C \lambda_1^{-1/r}(r)
			\|\nabla v\|_r 
			\leq 
			C \lambda_1^{-1/r}(r) \,
			\max_{u \in A} \|\nabla u\|_r <+\infty
			\quad \text{for any}~ v \in A_\varepsilon.
		\end{equation} 
		Moreover, $A_\varepsilon$ is closed in $\Wq$. 
		To show this, let us take any sequence $\{u_n\} \subset A_\varepsilon$ which strongly converges in $\Wq$ to a function $u_0 \in \Wq$. 
		By the compactness of $A_\varepsilon$ in $\Wr$ and the assumption $r<q$, we see that $u_0 \in \Wr \setminus \{0\}$ and $u_n \to u_0$ strongly in $\Wr$, and so $u_0 \in A_\varepsilon$, which gives the desired closedness of $A_\varepsilon$ in $\Wq$. 
		
		It remains to prove that $\gamma(A_\varepsilon; L^q)\ge k$.
		Consider the mapping $i: A_\varepsilon \subset \Wr \to L^q$ defined as $i(u)=u$. 
		This mapping is continuous. (This claim is trivial when $q$ does not exceed the critical Sobolev exponent $r^*$.)
		Indeed, let us take any sequence $\{u_n\} \subset A_\varepsilon$ which converges to some $u_0 \in A_\varepsilon$ strongly in $\Wr$.
		Since $A_\varepsilon$ is bounded in $\Wq$, we see that any subsequence of $\{u_n\}$ has a subsubsequence which converges to $u_0$ weakly in $\Wq$ and hence strongly in $L^q$. Therefore, the whole sequence $\{u_n\}$ also converges to $u_0$ strongly in $L^q$, and hence $i$ is continuous.
		Since $i$ is odd, we apply Lemma~\ref{lem:genus} to get $\gamma(A_\varepsilon; L^q)\ge \gamma(A_\varepsilon; \Wr) \geq k$.
		
		Thus, we conclude that $A_\varepsilon \in \mathcal{G}_k(q)$, which finishes the proof.
	\end{proof}

	\begin{remark}
		We refer to \cite[Sections~3 and~4]{HK} for several assumptions equivalent to/sufficient for \eqref{eq:WWW}. 
		In particular, if $\Omega$ is Lipschitz (in the case $N \geq 2$), then \eqref{eq:WWW} holds, see \cite[Remarks~(ii), p.~252]{HK}.
		Moreover, it is clear that \eqref{eq:WWW} is satisfied when $N=1$.
	\end{remark}

	\section{Relation between \texorpdfstring{$\lambda_k(q)$}{lambda-k(q)} and \texorpdfstring{$\beta_*$}{beta*}}\label{sec:appendixB}
	
	The aim of this section is to show that, for an appropriate domain $\Omega \subset \mathbb{R}^N$ with $N \geq 2$, the assumptions \ref{thm1:2}, \ref{thm1:2x}, \ref{thm1:3} of Theorem~\ref{thm1} can be satisfied for $k \geq 2$.
	Throughout this section, 
	in order to represent the dependence of quantities on a domain $\Omega$, 
	we will use the expanded notation $\beta_*(\Omega)$, $\beta_*(\alpha;\Omega)$, $\lambda_k(q;\Omega)$ for $\beta_*$, $\beta_*(\alpha)$, $\lambda_k(q)$, respectively.
	In particular, recall the definition of $\beta_*(\Omega)$, see \eqref{def:values}:
	\begin{equation*}
		\beta_*(\Omega)
		=
		\frac{\|\nabla \varphi_p\|_q^q}{\|\varphi_p\|_q^q}.
	\end{equation*}
	
	\begin{lemma}\label{lem:beta*>lambdak}
		Let $1<q<p$, $N \geq 2$, and $k \geq 2$.
		Then there exists a bounded $C^2$-smooth domain $\Omega \subset \mathbb{R}^N$ such that $\lambda_k(q;\Omega) < \beta_*(\Omega)$.
	\end{lemma}
	\begin{proof}
		Let us construct a beads-type domain $\Omega_\varepsilon$ as follows.
		Take $k$ points in $\mathbb{R}^N$ lying on the $x_1$-axis:
		\begin{equation}\label{eq:xj}
			z_j = (j, 0, \dots, 0), \quad j = 1,\dots,k.
		\end{equation}
		Denote by $\Omega_0$ the union over $j$ of open balls $B_r(z_j)$ of radius $r \in (0,1/2)$ centered at $z_j$.
		Thanks to the choice of $r$, these balls are \textit{disjoint}.
		Let us now consider a \textit{domain} $\Omega_\varepsilon$ as the union of $\Omega_0$ with a set of $k-1$ thin cylindrical type channels $T_{j,\varepsilon}$ of maximal width $\varepsilon>0$ subsequently connecting the balls $B_r(z_j)$ and $B_r(z_{j+1})$ in a $C^2$-smooth way, so that $\Omega_\varepsilon$ is of class $C^2$
		and $\Omega_{\varepsilon_1} \supset \Omega_{\varepsilon_2}$ provided $\varepsilon_1 > \varepsilon_2>0$, see Figure~\ref{fig:beads}. 
		For instance, one can take
		$$
		T_{j,\varepsilon}
		=
		\{
		(x_1,\dots,x_N) \in \mathbb{R}^N:~ 
		j+r-\varepsilon < x_1 < j+1-r+\varepsilon,~
		x_2^2+\dots+x_{N}^2 < g_\varepsilon^2(x_1)
		\},
		$$
		where $\{g_\varepsilon\}$ is an appropriate family of smooth positive functions in the interval $[1,k+1]$
		satisfying $\max_{s \in [1,k+1]}g_\varepsilon(s) \to 0$ as $\varepsilon \to 0$, and $g_{\varepsilon_1} > g_{\varepsilon_2}$ provided $\varepsilon_1 > \varepsilon_2>0$, so that $T_{j,\varepsilon_1} \supset T_{j,\varepsilon_2}$. 
		
		\begin{figure}[!ht]
			\center{
				\includegraphics[scale=0.65]{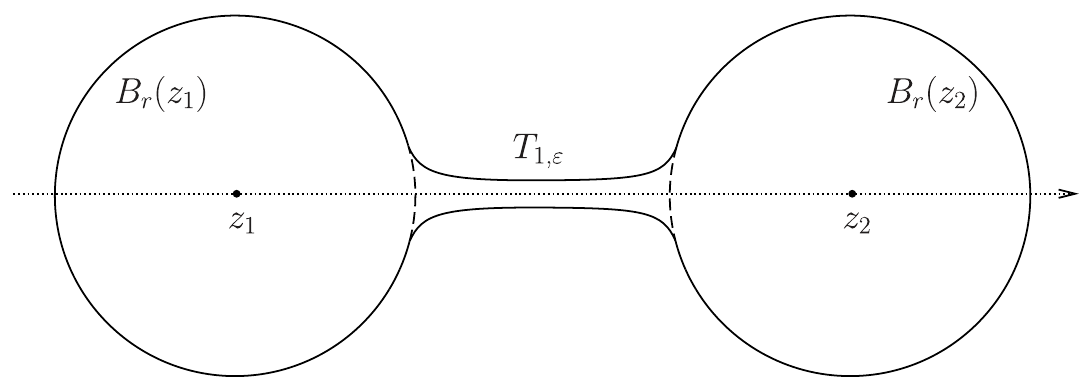}
			}
			\caption{A possible shape of $\Omega_\varepsilon$ for $k=2$ and some $\varepsilon>0$.}
			\label{fig:beads}
		\end{figure}

		Evidently, we have $\lambda_1(q;B_r(z_i)) = \lambda_1(q;B_r(z_j))$ for any $i,j \in \{1,\dots,k\}$, and when there is no ambiguity, we denote, for short,
		$$
		\lambda_1(q;B_r)
		=
		\lambda_1(q;B_r(z_j)).
		$$
		
		Our aim is to prove that $\lambda_k(q; \Omega_\varepsilon) < \beta_*(\Omega_\varepsilon)$ for any sufficiently small $\varepsilon>0$.
		First, let us show that 
		\begin{align}\label{eq:lambda1kx}
			\lambda_1(q;\Omega_\varepsilon) 
			\leq 
			\lambda_2(q;\Omega_\varepsilon) \leq
			\dots
			\leq
			\lambda_k(q;\Omega_\varepsilon) 
			\leq 
			\lambda_1(q;B_r).
		\end{align}
		For this purpose, denote by $\varphi_{q,j} \in \Wq(B_r(z_j)) \cap C^1(\overline{B_r(z_j)})$ the first eigenfunction corresponding to $\lambda_1(q;B_r(z_j))$. 
		Extending $\varphi_{q,j}$ by zero outside of $B_r(z_j)$, we have $\varphi_{q,j} \in \Wq(\Omega_\varepsilon)$.
		Let us consider the following subset of $\Wq(\Omega_\varepsilon)$:
		$$
		A = \{a_1 \varphi_{q,1} + \ldots + a_k \varphi_{q,k}:~ 
		a_1,\dots,a_k \in \mathbb{R},~
		|a_1|^q + \ldots + |a_k|^q = 1\}.
		$$
		It is clear that $A$ is symmetric and compact. Moreover, the mapping 
		$$
		a_1 \varphi_{q,1} + \ldots + a_k \varphi_{q,k} \mapsto (a_1,\dots,a_k)
		$$ 
		is continuous and odd, and hence $\gamma(A;\Wq) = k$, see \cite[Proposition~7.7]{rabinowitz}.
		Therefore, $A$ is an admissible set for the definition \eqref{eq:lambdak} of $\lambda_k(q;\Omega_\varepsilon)$.
		Since $\varphi_{q,i}$ and $\varphi_{q,j}$ have disjoint supports and equal norms when $i \neq j$, and 
		$$
		\intOe |\nabla  \varphi_{q,j}|^q \,dx
		=
		\lambda_1(q;B_r)
		\intOe |\varphi_{q,j}|^q \,dx,
		\quad j =1,\dots,k,
		$$ 
		we have
		\begin{align*}
			\lambda_k(q;\Omega_\varepsilon)
			&\leq
			\max_{u \in A} \frac{\intOe |\nabla u|^q \,dx}{\intOe |u|^q \,dx}
			\\
			&=
			\max \left\{ \frac{\intOe |\nabla (a_1 \varphi_{q,1} + \ldots + a_k \varphi_{q,k})|^q \,dx}{\intOe |a_1 \varphi_{q,1} + \ldots + a_k \varphi_{q,k}|^q \,dx}
			: |a_1|^q + \ldots + |a_k|^q = 1
			\right\}
			\\
			&=
			\max \left\{ 
			\frac{|a_1|^q \intOe |\nabla  \varphi_{q,1}|^q \,dx + \ldots + |a_k|^q \intOe |\nabla  \varphi_{q,k}|^q \,dx}{|a_1|^q \intOe |\varphi_{q,1}|^q \,dx + \ldots + |a_k|^q \intOe |\varphi_{q,k}|^q \,dx}
			: |a_1|^q + \ldots + |a_k|^q = 1
			\right\}\\
			&=
			\lambda_1(q;B_r),
		\end{align*}
		which establishes the chain of inequalities \eqref{eq:lambda1kx}.
		Exactly the same construction but in the space $\W(\Omega_0)$ instead of $\Wq(\Omega_\varepsilon)$ shows that
		\begin{equation}\label{eq:lambda1kB}
			\lambda_1(p;\Omega_0) 
			=
			\lambda_2(p;\Omega_0) 
			=
			\dots
			=
			\lambda_k(p;\Omega_0) 
			= 
			\lambda_1(p;B_r).
		\end{equation}
		
		\medskip
		As it can be seen from \eqref{eq:lambda1kx}, 
		in order to prove the lemma, it is sufficient to justify that
		\begin{equation}\label{eq:lambda1beta2}
			\lambda_1(q;B_r) < \beta_*(\Omega_\varepsilon)
			\quad \text{for any sufficiently small}~ \varepsilon>0.
		\end{equation}
		Hereinafter, we assume that $j \in \{1,\dots,k\}$ is fixed and we denote $B_r := B_r(z_j)$, for brevity.
		We will prove \eqref{eq:lambda1beta2} by noting that $\lambda_1(q;B_r) < \beta_*(B_r)$ (see Section~\ref{sec:notations}) and 
		showing that
		$$
		\beta_*(B_r)
		=
		\beta_*(\Omega_0)
		\leq 
		\beta_*(\Omega_\varepsilon) - o(1).
		$$
		Let us recall from Section~\ref{sec:notations} that $\beta_*(\Omega_\varepsilon)$ for any $\varepsilon \geq 0$ can be characterized as
		\begin{align}
			\notag
			\beta_*(\Omega_\varepsilon&)
			=
			\beta_*(\lambda_1(p;\Omega_\varepsilon); \Omega_\varepsilon)
			\\
			\label{def:beta*epsilon}
			&=
			\inf
			\left\{
			\frac{\intOe |\nabla u|^q \,dx}{\intOe |u|^q \,dx}: ~
			u \in \W(\Omega_\varepsilon) \setminus \{0\},~	
			\intOe |\nabla u|^p \,dx - \lambda_1(p;\Omega_\varepsilon) \intOe |u|^p \,dx \leq 0
			\right\}.
		\end{align}
		In particular, taking the first eigenfunction $\varphi_{p,j}$ corresponding to $\lambda_1(p;B_r)$ as a test function for \eqref{def:beta*epsilon} with $\varepsilon=0$ and using \eqref{eq:lambda1kB}, it is not hard to deduce that $\beta_*(\Omega_0) = \beta_*(B_r)$.
		
		Recall from Section~\ref{sec:notations} that the mapping $\alpha \mapsto \beta_*(\alpha;\Omega_\varepsilon)$ does not increase. 
		Therefore, noting that $\lambda_1(p;B_r) > \lambda_1(p;\Omega_\varepsilon)$ since $B_r \subsetneq \Omega_\varepsilon$, we get
		\begin{equation}\label{eq:betastart1}
			\beta_*(\lambda_1(p;B_r); \Omega_\varepsilon)
			\leq
			\beta_*(\lambda_1(p;\Omega_\varepsilon); \Omega_\varepsilon)
			\equiv
			\beta_*(\Omega_\varepsilon).
		\end{equation}
		Let us investigate the behavior of $\beta_*(\lambda_1(p;B_r); \Omega_\varepsilon)$ with respect to $\varepsilon$.
		The first eigenfunction $\varphi_{p,j}$ corresponding to $\lambda_1(p;B_r)$, being extended by zero outside of $B_r$ so that $\varphi_{p,j} \in \W(\Omega_\varepsilon)$, is an admissible test function for the definition \eqref{def:beta_*} of $\beta_*(\lambda_1(p;B_r); \Omega_\varepsilon)$, and hence
		\begin{equation}\label{eq:testbeta}
			\beta_*(\lambda_1(p;B_r); \Omega_\varepsilon)
			\leq 
			\frac{\intOe |\nabla \varphi_{p,j}|^q \,dx}{\intOe |\varphi_{p,j}|^q \,dx}
			=
			\frac{\int_{B_r} |\nabla \varphi_{p,j}|^q \,dx}{\int_{B_r} |\varphi_{p,j}|^q \,dx}
			=
			\beta_*(B_r) = \beta_*(\Omega_0).
		\end{equation}
		We want to show that
		\begin{equation}\label{eq:betastart2}
			\lim_{\varepsilon \searrow 0}
			\beta_*(\lambda_1(p;B_r); \Omega_\varepsilon)
			=
			\beta_*(\Omega_0).
		\end{equation}
		Suppose, by contradiction, that 
		$$
		\liminf_{\varepsilon \searrow 0}
		\beta_*(\lambda_1(p;B_r); \Omega_\varepsilon)
		< 
		\beta_*(\Omega_0).
		$$	
		Consequently, there exists a sequence $\{u_\varepsilon\}$ such that $u_\varepsilon \in \W(\Omega_\varepsilon)$,
		$\intOe |u_\varepsilon|^q \,dx = 1$,	
		\begin{equation}\label{eq:hles01}
			\intOe |\nabla u_\varepsilon|^p \,dx - \lambda_1(p;B_r) \intOe |u_\varepsilon|^p \,dx \leq 0,
		\end{equation}
		and
		\begin{equation}\label{eq:uepsbound1}
			\liminf_{\varepsilon \searrow 0}
			\intOe |\nabla u_\varepsilon|^q \,dx
			< 
			\beta_*(\Omega_0).
		\end{equation}
		
		By the construction, we have $\Omega_{\varepsilon_2} \subset \Omega_{\varepsilon_1}$ provided $0<\varepsilon_2 < \varepsilon_1$, which yields
		$\W(\Omega_{\varepsilon_2}) \subset \W(\Omega_{\varepsilon_1})$. 
		Therefore, extending $u_\varepsilon$ by zero outside of $\Omega_\varepsilon$, we get
		$u_\varepsilon \in  \W(\Omega_{\varepsilon_1})$ for a fixed $\varepsilon_1>0$ and all $\varepsilon \in (0,\varepsilon_1)$.
		In view of \eqref{eq:uepsbound1}, $\{u_\varepsilon\}$ is bounded in $\W(\Omega_{\varepsilon_1})$ and hence there exists $u_0 \in \W(\Omega_{\varepsilon_1})$ such that $u_\varepsilon \to u_0$ weakly in $\W(\Omega_{\varepsilon_1})$ and $\Wq(\Omega_{\varepsilon_1})$, strongly in $L^p(\Omega_{\varepsilon_1})$ and $L^q(\Omega_{\varepsilon_1})$, and a.e.\ in $\mathbb{R}^N$, up to a subsequence.
		
		By the construction, we have $u_0 = 0$ a.e.\ in $\mathbb{R}^N \setminus \overline{\Omega_0}$.
		Indeed, if we suppose that $u_0 \neq 0$ in some set $\mathcal{O} \subset \mathbb{R}^N \setminus \overline{\Omega_0}$ with positive measure $|\mathcal{O}|>0$, then 
		there exists $\varepsilon_0>0$ 
		and a channel $T_{l,\varepsilon_0}$ (which connects $B_r(z_l)$ with $B_r(z_{l+1})$) of maximal width $\varepsilon_0$ 
		such that $|\mathcal{O} \cap T_{l,\varepsilon_0}| > 0$.
		Since the channel shrinks to a line as $\varepsilon \to 0$, we see that $|\mathcal{O} \cap \left(\mathbb{R}^N \setminus T_{l,\varepsilon}\right)| > 0$ for any sufficiently small $\varepsilon>0$, which means that $|\mathcal{O} \cap \left(\mathbb{R}^N \setminus \Omega_\varepsilon\right)| > 0$.	
		However, this is impossible since $u_\varepsilon \in \W(\Omega_\varepsilon)$ and it is extended by zero outside of $\Omega_\varepsilon$, and $u_\varepsilon \to u_0$ a.e.\ in $\mathbb{R}^N$.
		Therefore, we have $u_0 \in \W(\Omega_0)$, see, e.g., \cite[Corollary~3.3]{tiba} for an explicit reference.	
		
		The strong convergence in $L^q(\Omega_{\varepsilon_1})$ yields
		\begin{equation}\label{eq:lem:uq=1}
			1
			=
			\intOe |u_\varepsilon|^q \,dx
			=
			\int_{\Omega_{\varepsilon_1}} |u_\varepsilon|^q \,dx
			\to
			\int_{\Omega_{\varepsilon_1}} |u_0|^q \,dx
			=
			\int_{\Omega_{0}} |u_0|^q \,dx
			~~\text{as}~ \varepsilon \searrow 0,
		\end{equation}
		and so $u_0 \not\equiv 0$ in $\Omega_{0}$.
		Applying also the weak convergence in $\W(\Omega_{\varepsilon_1})$ and the strong convergence in $L^p(\Omega_{\varepsilon_1})$, we obtain
		\begin{align*}
			\int_{\Omega_{0}} |\nabla u_0|^p \,dx - \lambda_1(p;B_r) \int_{\Omega_{0}} |u_0|^p \,dx	
			&=
			\int_{\Omega_{\varepsilon_1}} |\nabla u_0|^p \,dx - \lambda_1(p;B_r) \int_{\Omega_{\varepsilon_1}} |u_0|^p \,dx
			\\
			&\leq
			\liminf_{\varepsilon \searrow 0}
			\left(\int_{\Omega_{\varepsilon_1}} |\nabla u_\varepsilon|^p \,dx - \lambda_1(p;B_r) \int_{\Omega_{\varepsilon_1}} |u_\varepsilon|^p \,dx 
			\right)
			\\
			&=
			\liminf_{\varepsilon \searrow 0}
			\left(\intOe |\nabla u_\varepsilon|^p \,dx - \lambda_1(p;B_r) \intOe |u_\varepsilon|^p \,dx 
			\right)
			\leq 0,
		\end{align*}
		where the last inequality follows from \eqref{eq:hles01}.
		Thus, we have shown that $u_0$ is an admissible function for the definition \eqref{def:beta*epsilon} of $\beta_*(\Omega_0)$, and hence, using \eqref{eq:uepsbound1} and \eqref{eq:lem:uq=1}, we arrive at the following contradiction:
		\begin{align*}
			\beta_*(\Omega_0) 
			\leq
			\frac{\intOo |\nabla u_0|^q \,dx}{\intOo |u_0|^q \,dx}
			=
			\intOo |\nabla u_0|^q \,dx
			&\leq
			\liminf_{\varepsilon \searrow 0}
			\int_{\Omega_{\varepsilon_1}} |\nabla u_\varepsilon|^q \,dx
			\\
			&=
			\liminf_{\varepsilon \searrow 0}
			\intOe |\nabla u_\varepsilon|^q \,dx
			< 
			\beta_*(\Omega_0).
		\end{align*}
		Therefore, the convergence \eqref{eq:betastart2} is established.
		
		Combining now \eqref{eq:betastart1} and \eqref{eq:betastart2}, we see that
		$$
		\beta_*(\Omega_0) + o(1) \leq \beta_*(\Omega_\varepsilon).
		$$
		That is, we obtain the desired relation
		$$
		\lambda_1(q;B_r)  
		<
		\beta_*(B_r)
		=
		\beta_*(\Omega_0)
		\leq 
		\beta_*(\Omega_\varepsilon) - o(1),
		$$
		which justifies \eqref{eq:lambda1beta2}.
		Finally, combining \eqref{eq:lambda1beta2} and \eqref{eq:lambda1kx}, we complete the proof of the lemma.
	\end{proof}
	
	\begin{remark}\label{rem:astar}
		Essentially the same proof shows the existence of a smooth bounded domain $\Omega \subset \mathbb{R}^N$ such that $\lambda_l(p) < \alpha_*$ for given $N \geq 2$ and $l \geq 2$.
		We do not provide further details since this inequality can be observed already in the case $N=1$, see \cite[Lemma~A.3]{BobkovTanakaNodal1}.
	\end{remark}

	%%%%%%%%%%%%%%%%%%%%%%%%%%%%%%%%%%%%%%%%%%%%%%%%%%%%%%%%%%%	
	%%%%%%%%%%%%%%%%%%%%%%%%%%%%%%%%%%%%%%%%%%%%%%%%%%%%%%%%%%%
	\bigskip
	\noindent
	\textbf{Acknowledgments.}
	M.~Tanaka was supported by JSPS KAKENHI Grant Number JP~23K03170.

\end{document}